\documentclass[a4paper,11pt]{article}

\usepackage[pdftex]{graphicx}

\usepackage[colorlinks]{hyperref}
\DeclareGraphicsExtensions{.ps,.eps}
\usepackage{epsf}
\usepackage{psfrag}

\usepackage{verbatim}
\usepackage[normalem]{ulem}
\usepackage{fancybox}
\usepackage{textcomp}
\usepackage{amsmath}
\usepackage{amsfonts}
\usepackage{amssymb}
\usepackage{color}
\usepackage{amsthm}
\usepackage{mathrsfs}
\usepackage{eurosym}
\usepackage[T1]{fontenc}
\usepackage{bbm}
\evensidemargin \oddsidemargin
\newtheorem{theorem}{Theorem}[section]
\newtheorem{lemma}[theorem]{Lemma}
\newtheorem{proposition}[theorem]{Proposition}
\newtheorem{corollary}[theorem]{Corollary}
\newtheorem{definition}[theorem]{Definition}

\newcommand{\ep}{\varepsilon}
\title{On the Gibbs states of the noncritical Potts model on $\Z^2$}
\author{Loren Coquille, Hugo Duminil-Copin, Dmitry Ioffe and Yvan Velenik}
\date{\today}

\newcommand{\FK}{\mu_{\Lambda_n}^{\sff}}

\newcommand{\FKflower}{\mu_{\mathcal D}^{\text{flower}}}
\newcommand{\FKff}{\mu_{\mathcal D}^{\sff}}

\newcommand{\Cond}{\;  | \; {\rm Cond}_n[\sigma]}
\newcommand{\cond}{{\rm Cond}_n[\sigma]}

\newcommand{\R}{\mathbb{R}}
\newcommand{\Z}{\mathbb{Z}}

\newcommand{\ubbG}{\underline{\bbG}}

\renewcommand{\emptyset}{\varnothing}
\newcommand{\Angle}[2]{\measuredangle(#1,#2)}

\newcommand{\bfB}{\mathbf{B}}

\newcommand{\abs}[1]{\left| #1\right|}

\newcommand{\calC}{\mathcal{C}}
\newcommand{\calD}{\mathcal{D}}
\newcommand{\calE}{\mathcal{E}}
\newcommand{\calF}{\mathcal{F}}
\newcommand{\calG}{\mathcal{G}}

\newcommand{\calK}{\mathcal{K}}

\newcommand{\calM}{\mathcal{M}}

\newcommand{\calP}{\mathcal{P}}

\newcommand{\calR}{\mathcal{R}}
\newcommand{\calS}{\mathcal{S}}
\newcommand{\calT}{\mathcal{T}}

\newcommand{\calY}{\mathcal{Y}}

\newcommand{\frt}{\mathfrak{t}}

\newcommand{\frB}{\mathfrak{B}}

\newcommand{\bbG}{\mathbb{G}}

\newcommand{\bbN}{\mathbb{N}}

\newcommand{\bbP}{\mathbb{P}}

\newcommand{\bbR}{\mathbb{R}}

\newcommand{\bbV}{\mathbb{V}}

\newcommand{\bbZ}{\mathbb{Z}}
\newcommand{\sfw}{{\sf w}}
\newcommand{\sff}{{\sf f}}
\newcommand{\sfC}{{\sf C}}

\newcommand{\lb}{\left(}
\newcommand{\rb}{\right)}
\newcommand{\lbr}{\left\{}
\newcommand{\rbr}{\right\}}
\newcommand{\lra}{\leftrightarrow}
\newcommand{\slra}[1]{\stackrel{#1}{\longleftrightarrow}}

\newcommand{\dd}{{\rm d}}

\newcommand{\U}{{\mathbf U}_\tau}

\newcommand{\B}{{\mathbf B}}
\newcommand{\BB}{\hat{\mathbf B}}

\newcommand{\epp}{{\varepsilon^\prime}}

\newcommand{\step}[1]{\textsc{Step}\,#1}
\newcommand{\case}[1]{\textsc{Case}\,#1}

\newcommand{\be}[1]{\begin{equation}\label{#1}}
\newcommand{\ee}{\end{equation}}

\def\1{\ifmmode {1\hskip -3pt \rm{I}}
\else {\hbox {$1\hskip -3pt \rm{I}$}}\fi} %indicator

\newcommand{\smo}[1]{{\mathrm o}_{#1}(1)}

\newcommand{\df}{\stackrel{\Delta}{=}}

\begin{document}
\maketitle

\begin{abstract}
We prove that all Gibbs states of the $q$-state nearest neighbor Potts model on
$\mathbb{Z}^2$ below the critical temperature are convex combinations of the $q$
pure phases; in particular, they are all translation invariant. To achieve this
goal, we consider such models in large finite boxes with arbitrary boundary
condition, and prove that the center of the box lies deeply inside a pure phase
with high probability. Our estimate of the finite-volume error term is of
essentially optimal order, which stems from the Brownian scaling of fluctuating
interfaces. The results hold at any supercritical value of the inverse
temperature $\beta >\beta_c (q) = \log\lb 1+\sqrt{q}\rb$.
\end{abstract}


\section{Introduction}

\subsection{History of the problem}

Since the seminal works of Dobrushin and Lanford-Ruelle~\cite{Dob68,LanRue69},
the equilibrium states of a lattice model of statistical mechanics in the
thermodynamic limit --- the so-called Gibbs states --- are identified with the
probability measures $\mu$ that are solutions of the DLR equation,
\[
\mu(\cdot)=\int\mathrm{d}\mu(\omega)\gamma_\Lambda(\cdot\,\vert\,\omega),
\qquad\text{for all finite subsets $\Lambda$ of the lattice,}
\]
where the probability kernel $\gamma_\Lambda$ is the Gibbsian specification
associated to the system; see~\cite{Geo88}. Under very weak assumptions 
(at least for bounded spins),
it can be shown that the set $\mathcal{G}$ of all Gibbs states is a non-empty
simplex. The analysis of $\mathcal{G}$ is thus reduced to determining its
extremal elements. In general, this is a very hard problem which remains
essentially completely open in dimensions $3$ and higher, for any nontrivial
model, even in perturbative regimes.

The problem of determining all extremal Gibbs states amounts to understanding
all possible local behaviors of the system. Pirogov-Sina{\u\i}'s
theory~\cite{PirSin75,Zah84} often allows, at very low temperatures, to
determine the pure phases of the model, i.e., the extremal, translation
invariant (or periodic) Gibbs states, as perturbations of the corresponding
ground states.
However, it might be the case that suitable boundary conditions induce
interfaces resulting in the \emph{local} coexistence of different thermodynamic
phases. That such a phenomenon can occur was first proved for the
nearest-neighbor ferromagnetic (n.n.f.) Ising model on $\mathbb{Z}^3$ by
Dobrushin~\cite{Dob72}, by considering the model in a cubic box with $+$ spins
on the  top half boundary of the box and $-$ spins on the bottom half (the
so-called Dobrushin boundary condition). He proved that, at low enough
temperatures, the induced interface is rigid --- it is given by a plane with
local defects --- and the corresponding Gibbs state is extremal.

In two dimensions, the situation is very different. Gallavotti~\cite{Gal72}
proved, by studying the fluctuations of the corresponding interface, that the
Gibbs state of the (very low temperature) n.n.f. Ising model on $\mathbb{Z}^2$
obtained using the Dobrushin boundary condition is the mixture $\tfrac12(\mu^+ +
\mu^-)$, where $\mu^+$ and $\mu^-$ are the two pure phases of the Ising model.
This was refined by Higuchi~\cite{Hig79}, who proved that the interface, after
diffusive scaling, weakly converges to a Brownian bridge at sufficiently low
temperatures. These two results were then pushed to all subcritical temperatures
by, respectively, Messager and Miracle-Sole~\cite{MesMir77} and Greenberg and
Ioffe~\cite{GreIof05}.
A weaker but very simple and general proof of the non-extremality of the
state obtained using Dobrushin boundary condition can be found in~\cite{BLP79}.

The fact that the Dobrushin boundary condition gives rise to a translation
invariant Gibbs state is a strong indication that all Gibbs states of the
two-dimensional Ising model should be translation invariant: because of the
large fluctuations of the interfaces, a small box deep inside the system should
remain, with high probability, far away from any of the interfaces that are
induced by the boundary condition. Thus the possible local behaviors of the
system should correspond to the pure phases.

In the late 1970s, this phenomenology was established in the celebrated works of
Aizenman~\cite{Aiz80} and Higuchi~\cite{Hig81}, based on important earlier work
of Russo~\cite{Rus79}. They proved that
$\mathcal{G}=\{\alpha\mu^++(1-\alpha)\mu^-\,:\,0\leq \alpha\leq 1\}$ for the
n.n.f. Ising model on $\mathbb{Z}^2$. Their approaches relied on many specific
properties of the Ising model (in particular, GKS, Lebowitz and FKG inequalities
were used in the proof). A decade ago, Georgii and Higuchi~\cite{GeoHig00}
devised a variant of this proof with a number of advantages. In particular,
their version only relies on the FKG inequality (and some lattice symmetries),
which made it possible to obtain in the same way a complete description of Gibbs
states in several other models: the n.n.f. Ising model on the triangular and
hexagonal lattices, the antiferromagnetic Ising model in an homogeneous field
and the hard-core lattice gas. It should be emphasized that  all these works
deal directly with the infinite-volume system, and have only very weak
implications for large finite systems. In particular, the reasoning underlying
these arguments (taking the form of a proof by contradiction) remains far from
the heuristics of interfaces fluctuations.

A much more general result, restricted to very low temperatures, was established
by Dobrushin and Shlosman~\cite{DobShl85}. They proved that, under suitable
assumptions (finite single-spin space, bounded interactions, finite number of
periodic ground states), all Gibbs states are periodic, and in particular are
convex combinations of the pure phases corresponding to perturbations of the
ground states of the model. Their approach deals with finite systems and is
closer in spirit, if not fully in practice, to the above heuristics. Namely,
even though interface fluctuations play a central role in the approach
of~\cite{DobShl85}, the authors resort to crude low temperature surgery
estimates without developing a comprehensive fluctuation theory.

Very recently, a completely different approach to the Aizenman-Higuchi result
was developed by two of us~\cite{CoqVel10}. This new approach, although still
restricted to the n.n.f. Ising model on $\mathbb{Z}^2$, presents several
advantages on the former ones. Unlike \cite{DobShl85} it does  not require a
very low temperature assumption, and actually holds for all sub-critical
temperatures.  Furthermore,  it provides a quantitative, finite-volume version
of the Aizenman-Higuchi theorem, with the correct rate of relaxation. Another
interesting feature of the proof is that it closely follows the outlined
heuristics and, consequently, should be much more robust.

In the present work, we extend the approach of~\cite{CoqVel10} to n.n.f. Potts
models on $\mathbb{Z}^2$. As will be seen below, two major factors make the
proof substantially more difficult in this case. The first one is of a physical
nature: In all previous non-perturbative studies, there were only two pure
phases, and thus macroscopic interfaces were always line segments. In the Potts
model with $3$ or more states, there are more than two phases and, consequently,
interfaces are more complicated objects, elementary macroscopic interfaces being
trees rather than lines. The second difficulty is of a technical nature: The
positive association of Ising spins, manifested through the FKG inequality,
simplified many parts of the proof in~\cite{CoqVel10}. Unfortunately, this
property does not hold anymore in the context of general $q$-state Potts models.
We will therefore avoid this difficulty by reformulating the problem in terms of
the random-cluster representation.

\subsection{Statement of the results}

Let $\Omega=\{1,\ldots,q\}^{\Z^2}$ be the space of configurations. Let $\Lambda$
be a finite subset of $\Z^2$, and $\Lambda^c=\Z^2\setminus\Lambda$ be its
complement. The \emph{finite-volume Gibbs measure} in $\Lambda$ for the
$q$-state Potts model with boundary conditions $\sigma\in\Omega$ and at
inverse-temperature $\beta>0$ is the probability measure on $\Omega$ (with the
associated product $\sigma$-algebra) defined by
\[
\mathbb P_{\beta,\Lambda}^\sigma(\eta)
=
\begin{cases}
\frac{1}{Z^\sigma_{\beta,\Lambda}}{\rm e}^{-\beta H_\Lambda(\eta)}& \text{ if
$\eta_i=\sigma_i$ for all $i\in \Lambda^c$}\\
0 &\text{ otherwise},
\end{cases}
\]
where the normalization constant $Z^\sigma_{\beta,\Lambda}$ is the partition
function. The Hamiltonian in $\Lambda$ is given by
\[
H_\Lambda(\eta) = -\sum_{\substack{i\sim
j\\\{i,j\}\cap\Lambda\neq\emptyset}}\delta_{\eta_i,\eta_j}
\]
where $i\sim j$ if $i$ and $j$ are nearest neighbors in $\Z^2$. In the case of
pure boundary condition $i \in\lbrace 1,\ldots, q\rbrace$, meaning that
$\sigma_x=i$ for every $x\in \Lambda^c$, we denote the measure by $\mathbb
P_{\beta,\Lambda}^{(i)}$.

For an arbitrary subset $A$ of $\Z^2$, let $\calF_A$ be the sigma-algebra
generated by spins in $\Lambda$. A probability measure $\mathbb{P}$ on $\Omega$
is an \emph{infinite-volume Gibbs measure} for the $q$-state Potts model at
inverse temperature $\beta$ if and only if it satisfies the following DLR
condition:
\[
\mathbb{P}(\cdot|\mathcal F_{\Lambda^c})(\sigma)
=
\mathbb{P}^\sigma_{\beta,\Lambda}\qquad\text{ for $\mathbb{P}$-a.e. $\sigma$,
and all finite subsets $\Lambda$ of $\mathbb{Z}^2$}.
\]
Let $\calG_{q,\beta}$ be the space of infinite-volume $q$-state Potts measures.

Non-emptiness of $\calG_{q,\beta}$ can be proved constructively in this model.
For $i\in\{1,\ldots,q\}$, ${(\mathbb{P}^{(i)}_{\beta,\Lambda})}_\Lambda$
converges when $\Lambda\nearrow\Z^2$ (in particular, the limit does not depend
on the sequence of boxes chosen); this follows easily, e.g., from the random
cluster representation. We denote by $\mathbb{P}^{(i)}_\beta$ the corresponding
limit. It can be checked~\cite[Prop.~6.9]{GeoHagMae01} that the measures
$\mathbb{P}^{(i)}_\beta$ ($i=1,\ldots,q$) belong to $\calG_{q,\beta}$ and are
translation invariant.

\bigskip
When $\beta$ is less than the critical inverse temperature
$\beta_c(q)=\log(1+\sqrt{q})$~\cite{BefDum10}, it is known that there exists a
unique infinite-volume Gibbs measure (in particular $\mathbb
P^{(i)}_\beta=\mathbb P^{(j)}_\beta$ for every $i,j\in \{1,\ldots,q\}$). The
relevant values of $\beta$ for a study of $\calG_{\beta,q}$ are thus $\beta\ge
\beta_c(q)$.

In the present work, we extend ideas of~\cite{CoqVel10} in order to determine
all infinite-volume Gibbs measures for the $q$-state Potts models at inverse
temperature $\beta>\beta_c(q)$ on $\Z^2$. More precisely, we show that every
Gibbs state is a convex combination of infinite-volume measures with pure
boundary condition:
\begin{theorem}\label{main}
For any $q\ge2$ and $\beta>\beta_c(q)$,
\begin{equation} 
\label{eq:decomp-main}
\calG_{q,\beta}
=
\Bigl\{\sum_{i=1}^q\alpha_i\mathbb P_\beta^{(i)},\mbox{ where
}\alpha_i\ge0,\forall i\in\{1,\ldots, q\} \text{ and
}\sum_{i=1}^q\alpha_i=1\Bigr\}.
\end{equation}
\end{theorem}
\noindent
A straightforward yet  important corollary of this theorem is the fact that any
Gibbs state is invariant under translations.
\begin{corollary}
For any $q\ge2$ and $\beta>\beta_c(q)$, all elements of $\calG_{q,\beta}$ are
invariant under translations.
\end{corollary}
A second important corollary is the fact that the extremal Gibbs measures (also
called pure states) of the simplex $\calG_{q,\beta}$ are the infinite-volume
measures with pure boundary condition.
\begin{corollary}\label{cor:pure state}
For any $q\ge 2$ and $\beta>\beta_c(q)$, the extremal elements of
$\calG_{q,\beta}$ are the $\mathbb P_\beta^{(i)}, i\in\{1,\dots,q\}$.
\end{corollary}
This follows from Theorem~\ref{main}:
{ Define $\Delta (\beta )$ via 
$ \bbP_\beta^{(i)}(\eta_0\neq i) = (q-1)\Delta$, and    observe that
in the decomposition \eqref{eq:decomp-main} of 
$\bbP_\beta\in\calG_{q,\beta}$, the coefficient  $\alpha_i$ equals to 
$\frac{\bbP_\beta(\eta_0=i)-\Delta }{\bbP_\beta^{(i)}(\eta_0=i)-\Delta}.$
}
\medbreak
\noindent
Actually, our main result is stronger than Theorem~\ref{main}. As
in~\cite{CoqVel10}, we obtain a finite-volume, quantitative version of the
latter theorem, which, together with its proof, fully vindicates the heuristics
given above. For a measure $\mu$ and an integrable function $f$, we write
$\mu[f]=\int f\mathrm{d}\mu$.
\begin{theorem}\label{main theorem}
Let $q\ge 2$ and $\beta>\beta_c(q)$, and set $\Lambda_n=\Z^2\cap[-n,n]^2$. For
any $\varepsilon>0$ small enough, there exists $C_\varepsilon<\infty$ such that, for any
boundary condition $\sigma$ on $\partial \Lambda_n$, we can find
$\alpha_1^n,\ldots,\alpha_q^n\ge0$ depending on $(n,\sigma,\beta,q)$ only, such
that
\[
\bigl|\,\mathbb P_{\Lambda_n,\beta}^\sigma[g] - \sum_{i=1}^q\alpha_i^n\, \mathbb
P_{\beta}^{(i)}[g] \bigr|
\le
C_\varepsilon\|g\|_\infty n^{-\tfrac12 +14\varepsilon},
\]
for any measurable function $g$ of the spins in $\Lambda_{n^\varepsilon}$.
\end{theorem}
\noindent
Note that the error term is essentially of the right order (which is
$O(n^{-1/2})$); see~\cite{CoqVel10} for a proof of this claim when $q=2$.

\bigskip
The strategy of the proof is the following. We consider the conditioned
random-cluster measure on $\Lambda$ associated to the $q$-state Potts model with
boundary condition $\sigma$. Boundary conditions for the Potts model get
rephrased as absence of connections (in the random-cluster configuration)
between specified parts of the boundary of $\Lambda$. In other words, boundary
conditions for the Potts models correspond to conditioning on the existence of
dual-clusters between some dual-sites on the boundary. Note that the
conditioning can be very messy, since intricate boundary conditions correspond
to microscopic conditioning on existence of dual-clusters. It will be seen that
being a mixture of measures with pure boundary condition boils down to the fact
that, with high probability, no dual-cluster connected to the boundary reaches a
small box deep inside $\Lambda$ (which, in particular, implies that the same is
true for the Potts interfaces).

The techniques involved in the proof are two-fold. First, we use positivity of
surface tension in the regime $\beta>\beta_c$, which was proved
in~\cite{BefDum10}, in order to get rid of the microscopic mess due to the
conditioning and to show that, deep inside the box, the conditioning with
respect to $\sigma$ corresponds to the existence of macroscopic dual-clusters.
The second part of the proof consists in proving that these clusters are very
slim, and that they fluctuate in a diffusive way, so that the probability that
they touch a small box centered at the origin is going to zero as the size of
$\Lambda$ goes to infinity. The crucial step here is the use of the
Ornstein-Zernike {theory of sub-critical FK clusters  developed}
in~\cite{CamIofVel08}.

\subsection{Open problems}

Before delving into the proof, let us formulate some important open problems
related to the present study.

\noindent
$\blacktriangleright\;$\textit{Critical 2d Potts models.}
The behavior of two-dimensional $q$-state Potts models in the critical regime
$\beta=\beta_c(q)$ is still widely open. It is conjectured that there is a
unique Gibbs state when $q=3$ and $4$, but that, for $q\geq 5$, there is
coexistence at $\beta_c$ of $q+1$ pure phases: the $q$ low-temperature ordered
pure phases and the high temperature disordered phase. This is known to be true
when $q$ is large enough~\cite{LaaMesRui86,Mar86}. The extension of the latter
result to every $q>4$ remains a mathematical challenge. 

\noindent
$\blacktriangleright\;$\textit{Finite-range 2d models.}
The extension of the present result, even in the Ising case $q=2$, to general
finite-range interactions still seems out of reach today. There are, at least,
two main difficulties when dealing with such models: On the one hand, it is
difficult to find a suitable non-perturbative definition of interfaces (the
classical definitions used, e.g., in Pirogov-Sina{\u\i} theory become
meaningless once the temperature is not very low); on the other hand, interfaces
will not partition the system into (random) subsystems with pure boundary
conditions anymore, which implies that it will be necessary to understand
relaxation to pure phases from impure boundary conditions. Of course, the
general \emph{philosophy} of the approach we use should still apply.

\noindent
$\blacktriangleright\;$\textit{The question of quasiperiodicity.}
There is a general conjecture that two-dimensional models should always possess
a finite number of extremal Gibbs states, all of which are periodic. In
particular, this would imply that all Gibbs states are periodic, and thus that
a two-dimensional quasicrystal cannot exist (as an equilibrium state).

\noindent
$\blacktriangleright\;$\textit{Models in higher dimensions.}
Needless to say, the situation in higher dimensions is 
very different, due to the
existence of translation non-invariant states. Even in the very low-temperature
$3$-dimensional n.n.f. Ising model, the set of extremal Gibbs states is not
known. Note, however, that it has been proved, in the case of a $d$-dimensional
Ising model for any $d\geq 3$, that all \emph{translation invariant} Gibbs
states are convex combinations of the two pure phases at all
temperatures~\cite{Bod06}. A similar result also holds for large enough values
of $q$~\cite{Mar86}.

\subsection{Notations}
Each nearest-neighbor edge $e$ of $\Z^2$ intersects a unique dual edge of
$(\Z^2)^* = (\frac12 ,\frac12) +\Z^2$, that we denote by $e^*$. Consider a
subgraph $G=(V,E)$ of $\Z^2$, with vertex set $V$ and edge set $E$. If $E$ is a
set of direct edges, then its dual is defined by $E^* = \lbr e^*\, :\, e\in
E\rbr$.  Furthermore, if $G$ does not possess any isolated vertices, we can
define the dual $V^*$ as the endpoints of edges in $E^*$. Altogether, this
defines a dual graph $G^*=(V^*,E^*)$.

Let $\Lambda_n$ be the set of sites of $\Z^2\cap[-n,n]^2$ and $E_n$ be the set
of all nearest-neighbor edges of $\Lambda_n$. The dual graph is denoted by $\lb
\Lambda_n^* , E_n^*\rb$. For $m<n$, the annulus $\Lambda_n\setminus\Lambda_m$ is
denoted by $A_{m,n}$.

\noindent
The \emph{vertex-boundary} $\partial V$ of a graph $(V,E)$ is defined by
$\partial V = \lbr x\in V: \exists y\sim x\ \text{such that $y\not\in V$}\rbr$.

\noindent
{The \emph{exterior vertex-boundary} $\partial^{\rm ext} V$ of a graph $(V,E)$ is defined by
$\partial^{\rm ext} V = \cup_{x\in V}\lbr y\not\in V : y\sim x\ \rbr$.
}

\noindent
The \emph{edge-boundary} $\partial E$ of a graph $(V,E)$ is the set of edges
between two adjacent points of $\partial V$.

It will occasionally  be convenient to think about $\partial E_m$ as a closed
contour in $\bbR^2$ or, more generally, to think about subsets of $E$ (clusters,
paths, etc) in terms of their embedding into $\bbR^2$; we shall do it without
further comments in the sequel.

All constants in the sequel depend on $\beta$ and $q$ only. We shall use the
notation $f=O(g)$ if there exists $C=C(\beta,q)>0$ such that $|f|\le C|g|$. We
shall write $f=\Theta(g)$ if both $f=O(g)$ and $g=O(f)$.

%%%%%%%%%%%%%
%%%%%%%%%%%%%

\section{From Potts model to random-cluster model}

In this section, we relate Potts and random-cluster models. We will assume
throughout this article that the reader is familiar with the basic properties
of the Fortuin-Kasteleyn (FK) representation. A very concise and clear
exposition including  derivation 
of comparison inequalities could be found in \cite{ACCN88}. 
Mixing properties of random cluster measures were studied in \cite{Ale98, Ale04}. 
There is an extensive review \cite{GeoHagMae01} and
a book \cite{Gri06} on the subject. More recent results 
\cite{CamIofVel08, BefDum10} play an important  role in our approach.

\medskip 
\noindent
Let $G=(V(G),E(G))$ be a finite graph. An element $\omega\in\{0,1\}^{E(G)}$ is
called a configuration. An edge $e$ is said to be open in $\omega$ if
$\omega(e)=1$ and closed if $\omega(e)=0$. 
We shall work with two types of boundary conditions: $\sff$-free and 
$\sfw$-wired. 
Recall that the random-cluster
measure with edge-weight $p$ and cluster-weight $q$ on $G$ with 
$*$-boundary
condition ($*= \sff, \sfw$) is given by
\[
\mu_{G,p,q}^*(\omega)
=
\mu_{G}^{*}(\omega)=\frac{p^{\#\,\text{open edges}}(1-p)^{\#\,\text{closed
edges}}q^{\#_*\,\text{clusters}}}{Z_{G,p,q}^*},
\]
where $Z_{G,p,q}^*$ is a normalizing constant and a cluster is a maximal
connected component of the graph $(V(G),\{e\in E(G)\,:\, \omega(e)=1\})$. 
The number $\#_\sff\,\text{clusters}$ counts all the disjoint clusters, whereas 
the number $\#_\sfw\,\text{clusters}$ counts only those disjoint clusters which 
are not connected to the vertex boundary $\partial V$.

\subsection{Coupling with a supercritical random-cluster model on $(\Z^2)^*$}
We consider the $q$-state Potts model on the graph $(\mathbb Z^2)^*$ at inverse
temperature $\beta>\beta_c(q)$. As the parameters $\beta$ and $q$ will always
remain fixed, we drop them from the notation. Fix $\sigma\in\lbr 1, \dots,
q\rbr^{(\mathbb Z^2)^*}$. For each $n$, we define the Potts measure
$\bbP^\sigma_{\Lambda_n^*}$ on $\Lambda_n^*$ with boundary condition $\sigma$ on
the vertex boundary $\partial \Lambda_n^*$.

It is a classical result (see, \emph{e.g.}, \cite{ACCN88, GeoHagMae01}) that the Potts
model can be coupled with a random-cluster configuration in the following way.
From a configuration of spins $\eta\in\{1,\dots,q\}^{V(\Lambda_n^*)}$, construct
a percolation configuration $\omega^*\in\{0,1\}^{E_n^*}$ by setting each edge in
$E_n^*$ to be
\begin{itemize}
\item closed if the two end-points have different spins,
\item closed with probability ${\rm e}^{-\beta}$ and open otherwise if the two
end-points have the same spins.
\end{itemize}
The measure thus obtained is a random-cluster measure on $(\Z^2)^*$ with
edge-weight $p^*=1-{\rm e}^{-\beta}$, cluster-weight $q$ and wired boundary
condition on $\partial\Lambda_n^*$, conditioned on the following event, called
${\rm Cond}_n[\sigma]$: writing $S_i= \lbr x\in\partial \Lambda_n^*\, :\, \sigma
(x) =i\rbr$, the sets $S_i$ and $S_j$ are not connected by open edges in $
E_n^*$, for every $i\ne j$ in $\{1,\ldots,q\}$. We denote this measure by $
\mu_{\Lambda_n^*}^{\sfw}(\cdot \Cond)$. When there is no conditioning, the
random-cluster measure with wired (resp. free) boundary condition is denoted by
$\mu_{\Lambda_n^*}^{\sfw}$ (resp. $\mu_{\Lambda_n^*}^{\sff}$).

Reciprocally, the Potts measure can be obtained from $\mu_{\Lambda_n^*}^{\sfw}(\cdot\Cond)$ by assigning to every
cluster a spin in $\{1,\ldots,q\}$ according to the following rule:
\begin{itemize}
\item For every $i\in\{1,\ldots,q\}$, sites connected to $S_i$ receive the spin $i$,
\item The sites of a cluster which is not connected to $S_i$
  receive the same spin in $\{1,\ldots,q\}$ chosen uniformly at random, independently of the
spins of the other clusters.
\end{itemize}
Thanks to the connection between Potts measures and random-cluster measures, tools provided by the theory of
random-cluster models can be used in this context. Note that the parameters of the corresponding random-cluster measure
are supercritical ($p^*>p_c(q)$).

\subsection{Coupling with the subcritical Random-Cluster model on $\Z^2$}
Rather than working with the supercritical random-cluster measure on $(\Z^2)^*$, we will be working with its subcritical
dual measure on $\Z^2$ (this is the reason for choosing to define the Potts model on $(\Z^2)^*$).
There is a natural one-to-one mapping between $\lbr 0,1\rbr^{ E_n^*}$ and $\lbr 0,1\rbr^{ E_n}$.
Namely, set $\omega(e) = 1-\omega (e^*)$.
In this way, both direct and dual FK configurations are defined on the same probability space.
In the sequel, the same notation will be used for percolation events in direct and dual
configurations. For instance, $\omega\in {\rm Cond}_n [\sigma ]$ means that $\omega^*\in {\rm Cond}_n [\sigma ]$.
The corresponding direct FK measure is $\mu^{\sff}_{\Lambda_n} (\cdot\Cond )$.

It is well-known~\cite{ChaChaSch87} that this defines an FK measure with parameters $q$ and $p$ satisfying
$pp^*/[(1-p)(1-p^*)]=q$.

Since we are working with the low temperature Potts model, the random-cluster model on $(\Z^2)^*$ corresponds to
$p^*>p_c(q)$ so that the random-cluster model on $\Z^2$ is subcritical ($p<p_c(q))$.
For this measure, ${\rm Cond}_n
[\sigma ]$ is an increasing event which requires the existence of direct open paths disconnecting different dual
$S_i$-s.
This reduces the problem to the study of the stochastic geometry of subcritical
clusters.
In particular, this enables us to use known results on the subcritical model.

Let us recall the few properties we will be using in the next sections. First, there is a unique infinite-volume
measure, denoted $\mu_{\mathbb Z^2}$. Second, there is exponential decay of connectivities in the random-cluster model
with parameter $p<p_c(q)$. These two properties imply the following corollary.
\begin{proposition}
\label{reference subcritical}
There exists $c>0$ such that, for $n$ large enough and $2k\leq n
\leq m$,
\begin{gather*}\mu_{A_{k,n}}^\sfw(\text{there exists a crossing of $A_{k,n}$})\leq {\rm e}^{-cn},\\
\mu_{\Lambda_n}^\sfw(\text{there exists a cluster of cardinality }m\text{ in }\Lambda_{n/2})\leq {\rm e}^{-cm},
\end{gather*}
where a \emph{crossing} is a cluster of $A_{m,n}$ connecting the inner box to the outer box.
\end{proposition}
A cluster surrounding the inner box of $A_{m,n}$ inside  the outer box 
of $A_{m,n}$ is said to be a
\emph{circuit}. Note that the  existence of  a dual circuit is a complementary 
event to the existence of a crossing between the inner and outer boxes.

Proposition~\ref{reference subcritical} follows from the exponential decay of
connectivities proved for any $p<p_c(q)$ in \cite{BefDum10} \emph{together with}
the uniqueness of the infinite-volume measure (this is required to tackle wired
boundary conditions, see \cite[Appendix]{CamIofVel08} for details). The result
would not be true at criticality when $q$ is very large, despite the fact that
there is exponential decay for free boundary conditions.

\paragraph{Surface tension}
Surface tension in the supercritical dual model is the inverse correlation
length in the primal sub-critical FK percolation. Let $p<p_c(q)$.
The surface tension in direction $x$ is defined by
\[
\tau(x)
=
\tau_p(x)
=
-\lim_{k\to\infty}\frac1k\log\mu_{\Z^2}(0\leftrightarrow [kx]),
\]
where $y\leftrightarrow z$ means that $y$ and $z$ belong to the same connected
component. We will also refer to it as the $\tau$-distance. By
Proposition~\ref{reference subcritical}, $\tau$ is equivalent to the usual
Euclidean distance on $\bbR^d$. Furthermore, by~\cite{CamIofVel08} it is
strictly convex, and the following sharp triangle inequality of
\cite{Iof94, PfiVel99} holds: There exists $\rho =\rho (p) >0$ such that
\begin{equation}\label{eq:sharp-ti}
\tau(x)+\tau(y) -\tau(x+y)\geq \rho  (|x|+|y|-|x+y|) .
\end{equation}
Define ${\rm d}_\tau(A,B)=\sup_{a\in A}\inf_{b\in B} \tau(a-b)$ to be the
$\tau$-Hausdorff distance between two sets.

\subsection{Reformulation of the problem in terms of the subcritical
random-cluster model}

\begin{theorem}\label{main RC}
Fix $p<p_c(q)$ and let $\varepsilon\in(0,1)$. Then, uniformly in all boundary
conditions $\sigma$,
\begin{equation}
\mu_{\Lambda_n}^\sff \bigl({\mathsf C}\cap\Lambda_{n^{\varepsilon}}\neq\emptyset
\Cond\bigr)
=
O(n^{-\frac{1}{2} + 14\varepsilon})
\end{equation}
where $\mathsf C$ is the set of sites connected to the boundary 
$\partial\Lambda_n$.
\end{theorem}
The proof of this theorem will be the core of the paper. Before delving into the
proof, let us show how it implies Theorem~\ref{main theorem}.
\begin{lemma}\label{lemthm}
Let $\beta>\beta_c(q)$. Then,
\begin{equation}\label{mixed}
\mathbb P^\sff_{(\mathbb Z^2)^*}
=
\frac1q\sum_{i=1}^q\mathbb P^{(i)}_{(\mathbb Z^2)^*}.
\end{equation}
\end{lemma}
\begin{proof}
Fix $\beta>\beta_c$. Note that $\bbP^{(i)}_{(\bbZ^2)^*}$ can be defined via the
coupling with the random-cluster measure as follows. Let $\mu_{(\bbZ^2)^*}$ be
the unique infinite-volume random-cluster measure on $(\bbZ^2)^*$. Since
$p^*>p_c(q)$, this measure possesses a unique infinite cluster. The Potts
measure $\bbP^{(i)}_{(\bbZ^2)^*}$ is constructed by assigning spin $i$ to the
infinite cluster, and a spin chosen uniformly at random for each finite cluster,
independently of the spin of the other clusters. The Potts measure
$\bbP^\sff_{(\bbZ^2)^*}$ can also be constructed from $\mu_{(\bbZ^2)^*}$ by
assigning to each cluster (including the infinite one) a spin chosen uniformly
at random, independently of the spin of the other clusters. We
deduce~\eqref{mixed} immediately.

\smallskip
Note that in general, $\bbP^{(i)}_{(\bbZ^2)^*}$ is constructed from the
infinite-volume random-cluster measure $\mu_{(\Z^2)^*}^{\sfw}$ while
$\bbP^\sff_{(\bbZ^2)^*}$ is constructed from the infinite-volume random-cluster
measure $\mu_{(\Z^2)^*}^{\sff}$. Therefore, if these two measures are different,
\eqref{mixed} will not be valid. This is the case when $p=p_c(q)$ and $q$ is
large enough.
\end{proof}
\begin{lemma}
There exists $c>0$ such that, for any $n>0$ and any subdomain $\Omega^*$ of
$(\Z^2)^*$ containing $\Lambda_{2n}^*$,
\begin{equation}
\label{eq:lemthm}
\mathbb P^\sff_{\Omega^*}[g]
=
\mathbb P^\sff_{(\Z^2)^*}[g]+O(\|g\|_\infty {\rm e}^{-cn}),
\end{equation}
for any $g$ depending only on spins in $\Lambda_n^*$. The same holds for pure
boundary conditions $i\in\{1,\ldots,q\}$.
\end{lemma}
\begin{proof}
We treat the case of the free boundary condition. The other cases follow from
the same proof. Since $p<p_c(q)$, the random-cluster model on $\Z^2$ has
exponential decay of connectivities. Therefore, \cite[Theorem~1.7(ii)]{Ale04}
implies the so-called ratio strong mixing property for the dual random-cluster
model:
If a percolation  event $A$ depends on edges from $E_A$ and if 
$B$ depends on edges from $E_B$, then, 
\begin{equation}
 \label{eq:ratio-mixing}
 \left| \frac{\mu^{\sff}_{(\bbZ^2)^*}(A\cap
B)}{\mu^{\sff}_{(\bbZ^2)^*}(A)\mu^{\sff}_{(\bbZ^2)^*}(B)} -1\right| \leq 
 \sum_{e_A \in E_A , e_B\in E_B} {\rm e}^{-c\, \dd (e_A , e_B )}  ,
\end{equation}
where $\dd (e_A , e_B )$ is a distance between edges $e_A$ and $e_B$ (for instance 
the distance between their mid-points). 

Together with the observation that
$\mu^{\sff}_{\Omega^*}=\mu_{(\Z^2)^*}^{\sff}(\cdot|\omega(e)=0,\forall e\notin
E(\Omega^*))$, this leads to
\begin{equation}\label{abc}
\bigl\lvert\mu^{\sff}_{\Omega^*}[f]-\mu_{(\Z^2)^*}^{\sff}[f]\bigr\rvert
=
O\bigl( {\rm e}^{-cn}\mu_{(\Z^2)^*}^{\sff}[f]\bigr)
\end{equation}
for any function $f$ depending only on edges in $E_{3n/2}^*$. More generally,
let $F$ be the event that there does not exist an open crossing in the annulus
$A_{n,3n/2}$ (this corresponds to the existence of a dual circuit surrounding
the origin). {The complement $F^c$ of this} 
event has exponentially small probability by
Proposition~\ref{reference subcritical}. Consider a function $f$ depending a
priori on every dual edges, but with the property that $f\mathbf{1}_{F}$ is
measurable with respect to edges in $E_{3n/2}^*$. We immediately find that
\[
\mu^{\sff}_{\Omega^*}[f]
=
\mu^{\sff}_{\Omega^*}[f\mathbf{1}_F]+O\big(||f||_{\infty}\mu^{\sff}_{\Omega^*}
(F^c)\big)
=
\mu^ {\sff}_{\Omega^*}[f\mathbf{1}_F]+O(||f||_{\infty}{\rm e}^{-cn})
\]
and similarly for $\mu^\sff_{(\Z^2)^*}[f]$, so that~\eqref{abc} is preserved for
this class of functions.
\medbreak
Now, consider $g$ depending only on spins in $\Lambda_n^*$. Via the coupling
with the random-cluster model, $\mathbb P^\sff_{\Omega^*}[g]$ and $\mathbb
P^\sff_{(\Z^2)^*}[g]$ can be seen as $\mu^{\sff}_{\Omega^*}[f]$ and
$\mu^{\sff}_{(\Z^2)^*}[f]$ for a certain function $f$, depending a priori on
every edge, but for which $f\mathbf{1}_F$ depends on edges in $E_{3n/2}^*$ only
(on the event $F$, the dual connections between vertices of $\Lambda_n^*$ are
determined by edges in $E_{3n/2}^*$). We conclude that
\[
\big|\mathbb P^\sff_{\Omega^*}[g]-\mathbb P^\sff_{(\Z^2)^*}[g]\big|
=
\big|\mu^{\sff}_{\Omega^*}[f]-\mu^{\sff}_{(\Z^2)^*}[f]
\big|
=
O\big(||f||_{\infty}{\rm e}^{-cn} \big).
\]
\end{proof}
\begin{figure}[t]
\begin{center}
\includegraphics[width=0.80\textwidth]{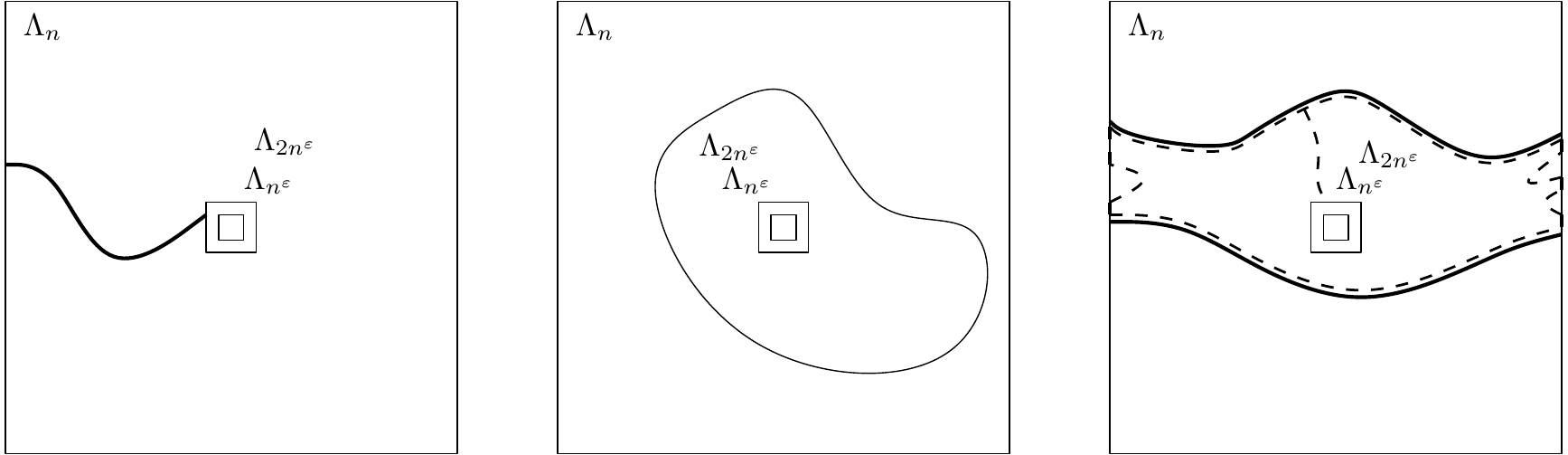}
\end{center}
\caption{On the left (resp. center, right), the event $ \mathcal E$ (resp.
$\calF^\sff$, $\calF^{(i)}$) is depicted.}
\label{fig:threecases}
\end{figure}
\begin{proof}[Proof of Theorem~\ref{main theorem}]
Fix $n>0$ and a boundary condition $\sigma$ on $\partial \Lambda_n$. Fix
$\varepsilon >0$ small.

We consider the coupling $(\eta,\omega)$ (the measure is denoted by $\mathbf P$)
with marginals $\mathbb P^\sigma_{\Lambda_n}$ and
$\mu^\sff_{\Lambda_n}(\cdot\Cond)$ described in the previous section. Let
$\mathcal E$ be the event that $\omega$ contains an open crossing in
$A_{2n^\varepsilon,n}$. Let $\mathcal F^\sff$ be the event that $\omega$
contains an open circuit in $A_{2n^\varepsilon ,n}$. Let $\mathcal F^{(i)}$ be
the event that $\omega$ contains neither an open crossing nor an open circuit in
$A_{2 n^\varepsilon ,n}$, and that $(\Lambda_{2n^\varepsilon})^*$ is connected
in the dual configuration to $S_i$. Note that
\[
\mathbf P(\mathcal E)
=
\mu^\sff_{\Lambda_n}(\mathcal E\Cond)
= O(n^{-\frac{1}{2} +14\ep}),
\]
by applying Theorem~\ref{main RC}.
\begin{itemize}
\item (conditioning on $\mathcal F^\sff$).
Let $\Gamma^*$  be the connected component of $\partial\Lambda_n^*$ in
$\omega^*$. Denote the connected component of $\Lambda_{2n^\varepsilon}^*$ in
$\Lambda_n^*\setminus \Gamma^*$  by $\Omega^*$. We have
$\Lambda_{2n^\varepsilon}^*\subset\Omega^*$. Conditioning on $\Gamma^*$ we
infer, using~\eqref{eq:lemthm} and~\eqref{mixed} that
\begin{align*}
\mathbf P\bigl(g\bigm|\mathcal F^\sff\bigr)
&=
\mathbf P\bigl( \mathbb P^\sff_{\Omega^*}[g] \bigm| {\mathcal F}^\sff\bigr) 
=
\mathbb P^\sff_{(\Z^2)^*}[g] + O(\|g\|_\infty {\rm e}^{-cn^{\varepsilon}}) \\
&=
\frac1q \sum_{i=1}^q\mathbb P^{(i)}_{(\Z^2)^*}[g] + O(\|g\|_\infty {\rm
e}^{-cn^{\varepsilon}}).
\end{align*}
\item (conditioning on $\mathcal F^{(i)}$).
In this case, let us condition on the connected cluster $\Gamma$ of $\partial
\Lambda_n$. We view $\Gamma$ as the set of bonds. Define $\Omega^*$ as the
connected component of $\Lambda_{2n^\varepsilon}^*$ in $\lb
E_n\setminus\Gamma\rb^*$. By  construction, $\Lambda_{2n^\varepsilon}^*\subset
\Omega^*$ and $\Omega^*\cap S_i \neq \emptyset$. Consequently,
using~\eqref{eq:lemthm} once again, we obtain
\[
\mathbf P\bigl(g\bigm|\mathcal F^{(i)}\bigr)
=
\mathbf P\bigl( \mathbb P^{(i)}_{\Omega^*}[g] \bigm| {\mathcal F}^{(i)}\bigr)
=
\mathbb P^{(i)}_{(\Z^2)^*}[g] + O(\|g\|_\infty {\rm e}^{-cn^{\varepsilon}}).
\]
\end{itemize}
By summing all these terms,
\begin{align*}
\mathbb P^\sigma_{\Lambda_n}[g]
&=
\mathbf P[g]
=
\mathbf P[g|\mathcal E]\, \mathbf P[\mathcal E] + \mathbf P[g|\mathcal F^\sff]\, \mathbf
P[\mathcal F^\sff] + \sum_{i=1}^q \mathbf P[g|\mathcal F^{(i)}]\mathbf P[\mathcal F^{(i)}]\\
&=
\sum_{i=1}^q \bigl( \tfrac1q \mathbf P[\mathcal F^\sff] + \mathbf P[\mathcal
F^{(i)}] \bigr) \mathbb P^{(i)}_{(\Z^2)^*}[g] + O(\|g\|_\infty n^{-\frac12
+14\varepsilon}),
\end{align*}
which implies the claim readily.
\end{proof}

\section{Macroscopic flower domains}

In the box $\Lambda_n$, the conditioning on ${\rm Cond}_n[\sigma]$ can be very
messy. Indeed, as we mentioned before, it forces the existence of open paths
separating the sets $S_i$. For instance, the number of such paths forced by an
alternating boundary condition $1,2,\dots,q,1,2,\dots$ is necessarily of order
$n$.

We first show that, no matter what the boundary condition $\sigma$ is, with high
probability  only a bounded number of such interfaces is capable of reaching an
inner box $\Lambda_m$, where $m$ is a fraction of $n$.  Furthermore, we shall
argue that the number of sites in $\partial\Lambda_m$ which are connected to the
original $\partial\Lambda_n$ is uniformly bounded.  In terms of the original
Potts model, this corresponds to the existence, with high probability, of a
domain including the box $\Lambda_m$ for which the boundary condition contains a
uniformly bounded number of spin changes. This will be called a flower domain
below.

\subsection{Definition of flower domains}

Let $m<n$. For a configuration $\omega$, let $\sfC_{m,n}=\sfC_{m,n}(\omega)$ be
the set of sites connected to $\partial\Lambda_n$ in $\omega\cap(E_n\setminus
E_m)$.  Define the set of \emph{marked vertices} by
\[
\bbG_{m,n}
=
\bbG_{m,n}(\omega)
=
\sfC_{m,n}\cap\partial\Lambda_{m}.
\]
The set $\bbG_{m,n}\cup\lb\Lambda_n\setminus\sfC_{m,n}\rb $ may have several
connected components, exactly one of them containing  $\Lambda_m$. Let us call
the latter the \emph{flower domain} $\calD_{m,n}=\calD_{m,n}(\omega)$
\emph{rooted at $m$}. Note that  $\bbG_{m,n} = \partial
\calD_{m,n}\cap\partial\Lambda_m$, that is marked sites are unambiguously
determined by the corresponding flower domains.

Fix a configuration $\omega$. Let $\calC=\sfC_{m,n}(\omega)$ and let
$\calD=\calD_{m,n}(\omega)$ be the corresponding flower domain. Let also
$\bbG=\bbG_{m,n}(\omega)$. By construction, the restriction of the conditional
measure $\mu^{\sff}_{\Lambda_n}(\, \cdot\,  | \sfC_{m,n} =\calC)$ to $\lbr
0,1\rbr^{\calE_{\calD}}$, where $\calE_{\calD}$ is the set of edges of $\calD$,
is the FK measure with free boundary conditions on $\partial\calD\setminus\bbG$
and wiring between sites of $\bbG$ inherited from connections in $\calC$. We
denote this restricted conditional measure as $\FKflower$. We also set
$\sfC_{\bbG}$ for the connected component of $\bbG$ in the restriction of 
$\omega $ to $\calE_{\calD}$.

\begin{figure}
\begin{center}
\includegraphics[width=0.40\textwidth]{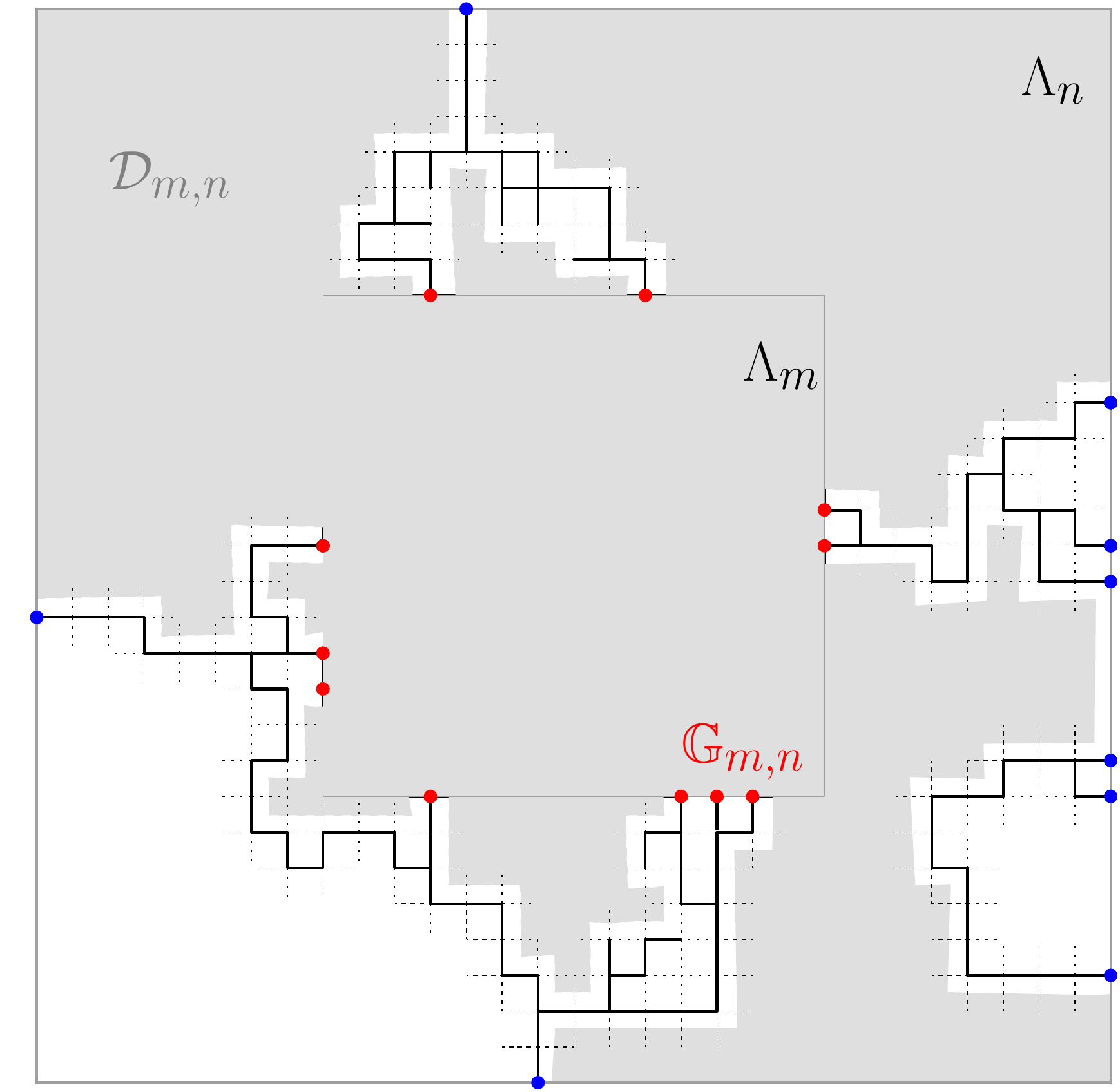}
\end{center}
\caption{Description of a flower domain $\mathcal D_{m,n}$ (light grey area).
The blue points are locations of spin changes (i.e. separation between sets
$S_i$), the red points constitute $\mathbb G_{m,n}$, the solid black lines in
the annulus $\Lambda_n\backslash\Lambda_m$ constitute $\sfC_{m,n}$.}
\label{fig:flowerdomain}
\end{figure}

\subsection{Cardinality of $\bbG_{m,n}$}

Flower domains have typically small sets $\bbG_{m,n}$, as the following
proposition shows.
\begin{proposition}
\label{flowers-ndelta}
There exists $M>0$ such that for any $\delta>0$
\begin{equation}\label{non-regPetals}
\mu_{\Lambda_n}^\sff\Bigl( \exists m\in\bigl[\tfrac{\delta n}{3},\delta
n\bigr]:|\bbG_{m,n}| \le M \Bigm| {\rm Cond}_n[\sigma]\Bigr) \ge 1-{\rm
e}^{-\delta n} ,
\end{equation}
uniformly in $\sigma$ and $n$ sufficiently large.
\end{proposition}
The notation $M$ will now be reserved for an integer $M>0$ satisfying the
previous proposition. We shall prove this Proposition for $\delta =1$; the
general case follows by a straightforward adaptation.
\begin{definition}
Let $\calE_r$ be the event that there exist $r$ disjoint crossings of
$A_{n/3,n/2}$.
\end{definition}
\begin{lemma}\label{lem:1}
For all $r\geq 1$ and $n>0$,
\[
\FK(\calE_r)\leq {\rm e}^{-crn},
\]
where $c>0$ is defined in Proposition~\ref{reference subcritical}.
\end{lemma}
\begin{proof}
We prove that for all $r\geq 1$ and $n>0$,
\begin{equation}\label{ab}
\FK(\calE_r)
\leq
\bigl(\mu^{\sfw}_{A_{n/3,n/2}}(\calE_1)\bigr)^{r}.
\end{equation}
The conclusion will then follow easily, since Proposition~\ref{reference
subcritical} implies that $\mu^{\sfw}_{A_{n/3,n/2}}(\calE_1)\leq \exp(-cn)$.

In order to prove~\eqref{ab}, we proceed by induction. First, note that $\FK$
restricted to $A_{n/3,n/2}$  is stochastically dominated by
$\mu_{A_{n/3,n/2}}^{\sfw}$.

Let $r\geq 1$ and consider $\FK(\calE_{r+1}|\calE_r)$. We number the vertices of
$\partial\Lambda_n=\{x_1,\ldots,x_{4n+4}\}$ in clockwise order, starting at the
bottom right corner. Let $k$ be the smallest number such that there are $r$
crossings among the clusters containing $x_1,\ldots,x_k$. Denote by $\calS$ the
union of these clusters (which may contain isolated vertices). Observe that all
edges in $A_{n/3,n/2}\setminus\calS$ which are incident to vertices of $\calS$
are closed. Therefore, the conditional measure  
$\FK(\cdot_{|A_{n/3,n/2}\setminus\calS} |\calS)$ is stochastically dominated by
$\mu^\sfw_{A_{n/3,n/2}}(\cdot_{|A_{n/3,n/2}\setminus\calS})$. 
{In both instances above, the symbol $\nu(\cdot_{|B} )$ means the restriction 
of $\nu$ to edges of the graph with the vertex set $B$.} 
As a result, the
probability, under $\FK(\cdot_{|A_{n/3,n/2}\setminus\calS} |\calS)$, that 
there exists a crossing of
$A_{n/3,n/2}$ is smaller than $\mu^{\sfw}_{A_{n/3,n/2}}(\calE_1)$. We obtain
\begin{align*}
\FK(\calE_{r+1})
&=
\FK (\calE_{r+1}|\calE_r)\FK(\calE_r)
=
\FK [\FK(\calE_{r+1}|\calS)]\FK(\calE_r)\\
&\leq
\mu^{\sfw}_{A_{n/3,n/2}}(\calE_1)\FK(\calE_{r})\leq\mu^{\sfw}_{A_{n/3,n/2}}
(\calE_1)^{r+1}.
\end{align*}
\end{proof}

\begin{proof}[Proof of Proposition~\ref{flowers-ndelta}]
Obviously,
\begin{equation}
\label{slim}
\mu_{\Lambda_n}^\sff\bigl(\forall m\in\left[\tfrac n3,\tfrac
n2\right]:|\bbG_{m,n}|> M \bigm| {\rm Cond}_n[\sigma]\bigr)
\leq
\frac{\mu_{\Lambda_n}^\sff\bigl(\forall m\in[\tfrac n3,\tfrac n2]:|\bbG_{m,n}|>
M\bigr)}{\mu_{\Lambda_n}^\sff({\rm Cond}_n[\sigma])}.
\end{equation}
Let us bound from below the denominator of~\eqref{slim}. If all the edges  of
$\partial  E_n$ are open, then ${\rm Cond}_n [\sigma]$ occurs. Moreover, the
measure $\FK$ stochastically dominates independent Bernoulli edge percolation on
$\lbr 0,1\rbr^{ E_n}$ with $\tilde{p} = p/(p+ (q-1)p)$,
see~\cite[Theorem~4.1]{ACCN88}. We deduce
\begin{equation}
\label{lower bound cleaning}
\FK({\rm Cond}_n[\sigma])\geq \FK (\mbox{\rm all the edges in $\partial E_n$ are
open}) \geq \tilde{p}^{8n}.
\end{equation}
Let us now bound from above the numerator of~\eqref{slim}. First,
\[
\FK\bigl(\forall m\in\bigl[\tfrac n3,\tfrac n2\bigr]:|\bbG_{m,n}| > M\bigr)
\leq
\FK\bigl(|\sfC_{n/3,n}\cap A_{n/3,n/2}|\geq M n/6\big).
\]
Fix $R>0$. If $|\sfC_{n/3,n}\cap A_{n/3,n/2}|\geq M n/6$, either $A_{n/3,n/2}$
contains more than $R$ crossings or one of the crossings has cardinality larger
than $M n/(6R)$. Proposition~\ref{reference subcritical} implies that the
probability of having clusters with size larger than $M n/(6R)$ in
$\Lambda_{n/2}$ is smaller than $\exp[-c M n/(6R)]$ for $n$ large enough.
Lemma~\ref{lem:1} together with~\eqref{slim} implies that, for $n$ large enough,
\[
\FK\bigl(\forall m\in\bigl[\tfrac n3,\tfrac n2\bigr]:|\bbG_{m,n}| > M \bigm|
{\rm Cond}_n[\sigma] \bigr)
\leq
\tilde p^{-8n}[{\rm e}^{-cRn}+{\rm e}^{-c Mn/(6R)}]\leq {\rm e}^{-n},
\]
provided that $R$ and $M$ be sufficiently large.
\end{proof}

\subsection{Reduction to FK measures on flower domains with free boundary
condition} \label{sec:redu}

We define
\begin{equation}\label{eq:Mn}
\calM_n
=
\max \{ m\leq n : \abs{\bbG_{m,n }}\leq M \},
\end{equation}
where the maximum is set to be equal to $\infty$ if there is no $m\le n$ such
that $\abs{\bbG_{m,n }}\leq M$. With this notation, we actually proved that
$\calM_n \in [\tfrac{n}{3}, n]$ with probability {bounded below by} 
$1-{\rm e}^{-{n} }$.

\smallskip
Let $\calC$ be  a possible realization of $\sfC_{m,n}$ and $\calD=\calD_{m,n}$
be the corresponding flower domain. The restriction of $\FK\lb\cdot~|~\calM_n =m
;\sfC_{m, n}=\calC\rb$ {to $\calD$} is $\FKflower$. Furthermore,
\[
\cond\cap\{ \calM_n=m\} \cap \{ \sfC_{m,n} = \calC \}
\]
is a product event $\Omega_{\sigma,\calC}\times \{ \calM_n= m ;
\sfC_{m,n}=\calC\}$, where $\Omega_{\sigma ,\calC}\subset \lbr
0,1\rbr^{\calE_{\calD }}$. Then
\begin{equation}\label{eq:reduction1}
\FK(\sfC\cap\Lambda_{n^\varepsilon}\neq\emptyset \Cond;\calM_n=m;\sfC_{m,
n}=\calC)
=
\FKflower ( \sfC_{\bbG} \cap\Lambda_{n^\varepsilon}\neq\emptyset \,|\,
\Omega_{\sigma ,\calC}) .
\end{equation}
The event $\Omega_{\sigma ,\calC}$ has an obvious  structure. It corresponds to
the existence of certain connections between different sites of $\bbG =
\calC\cap\partial\Lambda_m = \calD\cap\partial\Lambda_m$.  More precisely, let
$\calP_{\bbG}$ be the collection of different partitions of $\bbG$. Elements of
$\calP_\bbG$  are of the form $\underline{\bbG} = \lb \bbG_1, \dots
,\bbG_\ell\rb$. Define
\[
\Omega_{\underline{\bbG}}
=
\bigcap_{i}\bigcap_{u,v\in\bbG_i}\{ u\leftrightarrow v\} \subset
\{0,1\}^{\calE_{\calD}} .
\]
Let us say that a partition $\underline{\bbG}$ is compatible with 
$\Omega_{\sigma ,\calC}$ if $\Omega_{\underline{\bbG}}
\subseteq \Omega_{\sigma ,\calC}$.
Note that we do not rule out that some elements $\bbG_i$ of a partition
$\underline{\bbG}$ are singletons.  If $\bbG_i$ is a singleton, then
$\bigcap_{u,v\in\bbG_i}\{ u\leftrightarrow v\}$ is, of course, a sure event,
which could be dropped from the definition of $\Omega_{\underline{\bbG}}$. In
other words,  only non-singleton elements of $\underline{\bbG}$ are relevant for
$\Omega_{\underline{\bbG}}$. Also note that the events
$\Omega_{\underline{\bbG}}$ do not have to be disjoint.  Still, for any
$\sigma$,
\[
\Omega_{\sigma ,\calC}
=
\bigcup_{\underline{\bbG}\in\calP^\prime_\bbG} \Omega_{\underline{\bbG}},
\]
where the  set $\calP^\prime_\bbG$ corresponds to  partitions which are
compatible with the occurrence of the event $\Omega_{\sigma,\calC}$, and which
are \emph{maximal} in the sense that one cannot find a finer partition which 
would be still
compatible with $\Omega_{\sigma,\calC}$.

The previous section implies the following reduction, which we will now consider
for the rest of this work.
\begin{proposition}
\label{prop:reduction2}
Fix $\delta >0$.  Then, writing $B_k$ for the $k^{th}$ Bell number, which counts
the number of partitions of a set of $k$ elements,
\begin{equation}\label{eq:reduction2}
\FK \bigl( \sfC\cap\Lambda_{n^\varepsilon}\neq\emptyset \bigm| \cond\bigr)
\leq
{\rm e}^{-\delta n} +
B_{M} q ^{M} \max\FKff \bigl( \sfC_\bbG \cap
\Lambda_{n^\varepsilon}\neq\emptyset \bigm| \Omega_{\underline{\bbG}}\bigr) ,
\end{equation}
for all boundary conditions $\sigma$ and $n$ sufficiently large. The above
maximum is over all flower domains $\calD$ rooted at $m\in [\tfrac n3, n]$ with
at most $\abs{\bbG}\leq M$ marked points, and over all partitions
$\underline{\bbG}\in \calP'_\bbG$.
\end{proposition}
Above the term $q^{M}$ comes from the fact that the elements of $\bbG$ are
possibly wired together. It then bounds the Radon-Nikodym derivative between
measures $\FKflower$ and $\FKff$. The quantity $B_{M}$ bounds from above the
number of sub-partitions of $\bbG$ (the events $\Omega_{\underline{\bbG}}$ being
not necessarily disjoint).

\section{Macroscopic structure near the center of the box}

This section studies the macroscopic structure of the set $\sfC$ of sites
connected to the boundary of $\Lambda_n$. Its main result,
Proposition~\ref{tripod} below, implies that on a sufficiently small scale
$\delta >0$, the intersection $\sfC\cap\Lambda_k$ for boxes with
$k\in[\frac{\delta n}{3},\delta n]$
is with an overwhelming probability either empty, or close to a segment, or
close to a tripod (three segments coming out from a point).

Before starting, note that Proposition~\ref{prop:reduction2} enables us to
restrict attention to a flower domain $\calD=\calD_{m,n}$ with $m\in[\frac
n3,n]$ and $\abs{\bbG_{m,n}}\le M$. We set $\bbG=\bbG_{m,n}$. We now fix this
flower domain and work under $\FKff \lb \cdot ~\big|~
\Omega_{\underline{\bbG}}\rb $ for some $\underline \bbG\in \calP_{\bbG}'$. All
constants in this section are independent of $\calD_{m,n}$ and $\ubbG$ as long
as $|\bbG_{m,n}|\le M$. We will often recall this independence by using the
expression ``uniformly in $(\calD,\ubbG)$ with $|\bbG|\le M$''.
\bigbreak

Define $\sfC_{k,\bbG}$ to be the {set of edges connected to $\bbG$ in
$\calD\setminus\Lambda_k$ (which can consist of several connected components)}.
Note that $\bbG_{k,n}=\sfC_{k,n}\cap\partial \Lambda_k=\sfC_{k,\bbG}\cap\partial
\Lambda_k$. Given $v_1,v_2\in\R^2$, we define $[v_1,v_2]$ to be the line segment
with endpoints $v_1$ and $v_2$, and $\Angle{v_1}{v_2}$ to be the angle between
$v_1$ and $v_2$, seen as vectors in the plane. We refer to
Fig.~\ref{fig:3situations} for an illustration of the following definitions.

\begin{definition}\label{E1E2E3}
For $k< m$, $\nu >0$ and $\ell =1,2,3$, let us say that $E_{\nu ,k}^\ell \subset
\lbr 0,1\rbr^{\calE_\calD }$ occurs if S$\ell$ below  happens:
\begin{itemize}
\item[S1.] $\bbG_{k,n}=\emptyset$.
\item[S2.] $\bbG_{k,n} = \bbV_{k,n}^1\cup \bbV_{k,n}^2$,
where $\bbV_{k,n}^1 , \bbV_{k,n}^2$ are two disjoint sets of $\tau$-diameter
less than or equal to $\nu k$. Moreover,
\begin{itemize}
\item Each of the sets $\bbV_{k,n}^1$ and $\bbV_{k,n}^2$ is  connected in
$\sfC_{k, \bbG}$.
\item For any two vertices $v_i\in \bbV_{k,n}^i$; $i=1,2$, we have $[v_1
,v_2]\cap \Lambda_{k/2}\neq \emptyset$.
\end{itemize}
\item[S3.] $\bbG_{k,n} = \bbV_{k,n}^1\cup \bbV_{k,n}^2 \cup \bbV_{k,n}^3$, where
$\bbV_{k,n}^1$, $\bbV_{k,n}^2$ and $\bbV_{k,n}^3$ are disjoint sets with
$\tau$-diameter less than or equal to $\nu k$. Moreover,
\begin{itemize}
\item Each of the sets $\bbV_{k,n}^1$, $\bbV_{k,n}^2$ and $\bbV_{k,n}^3$ is
connected in $\sfC_{k ,\bbG}$,
\item For any choice of $v_i\in  \bbV_{k,n}^i;\ i=1,2,3$. there exists $x\in
\Lambda_{k/2}$ such that $\calT  = \lbr v_1 ,v_2 ,v_3 ; x\rbr$ is a Steiner
tripod  (see Definition~\ref{rem:Tripod} below). In particular, as it follows
from P2 of Proposition~\ref{prop:Steiner} below, $\Angle{v_i - x}{v_j -x} >
\frac{\pi}{2} +\eta$ for every $i\neq j$ .
\end{itemize}
\end{itemize}
\end{definition}

\begin{figure}
\begin{center}
\includegraphics[width=0.30\textwidth]{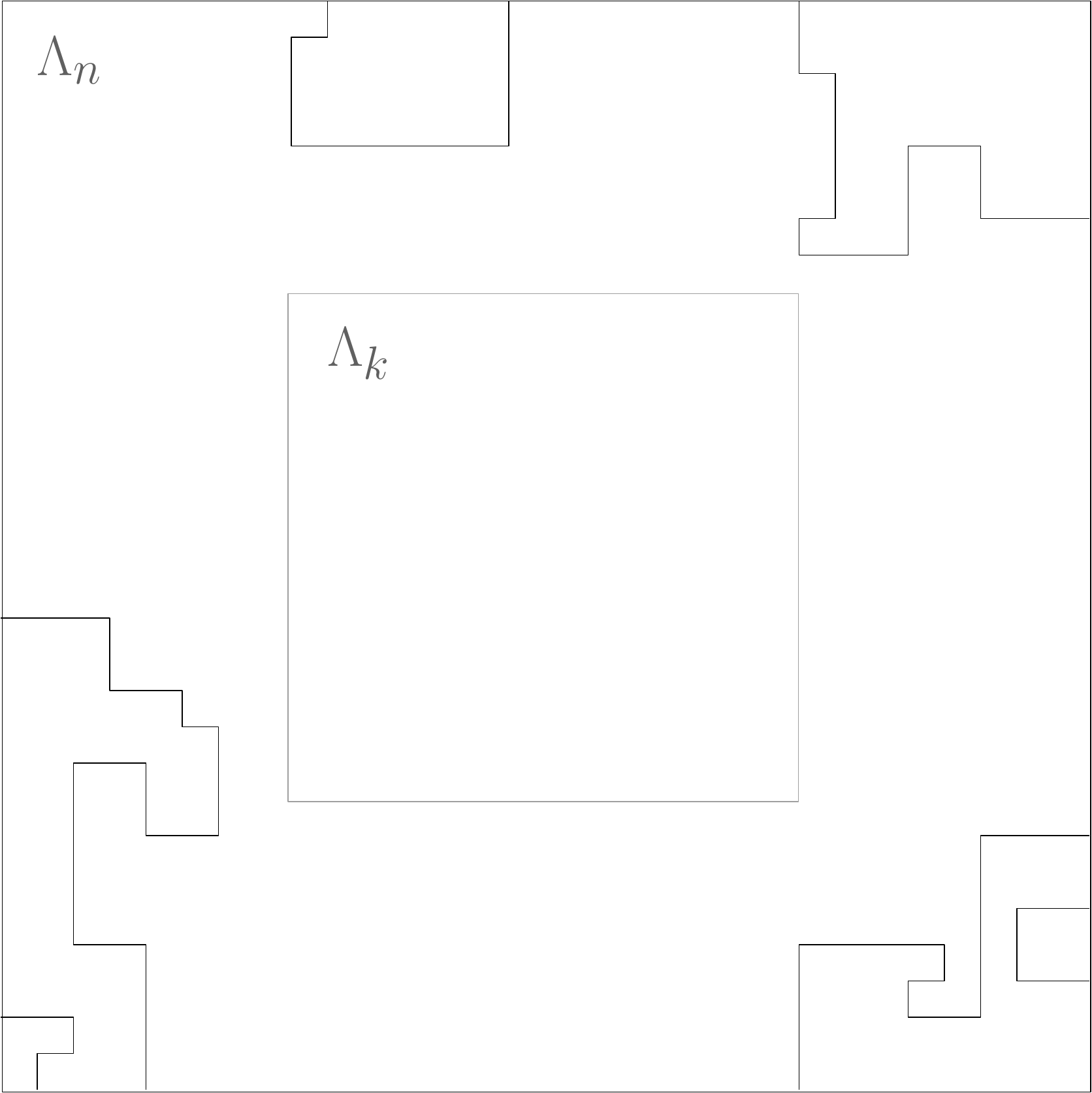}  \quad
\includegraphics[width=0.30\textwidth]{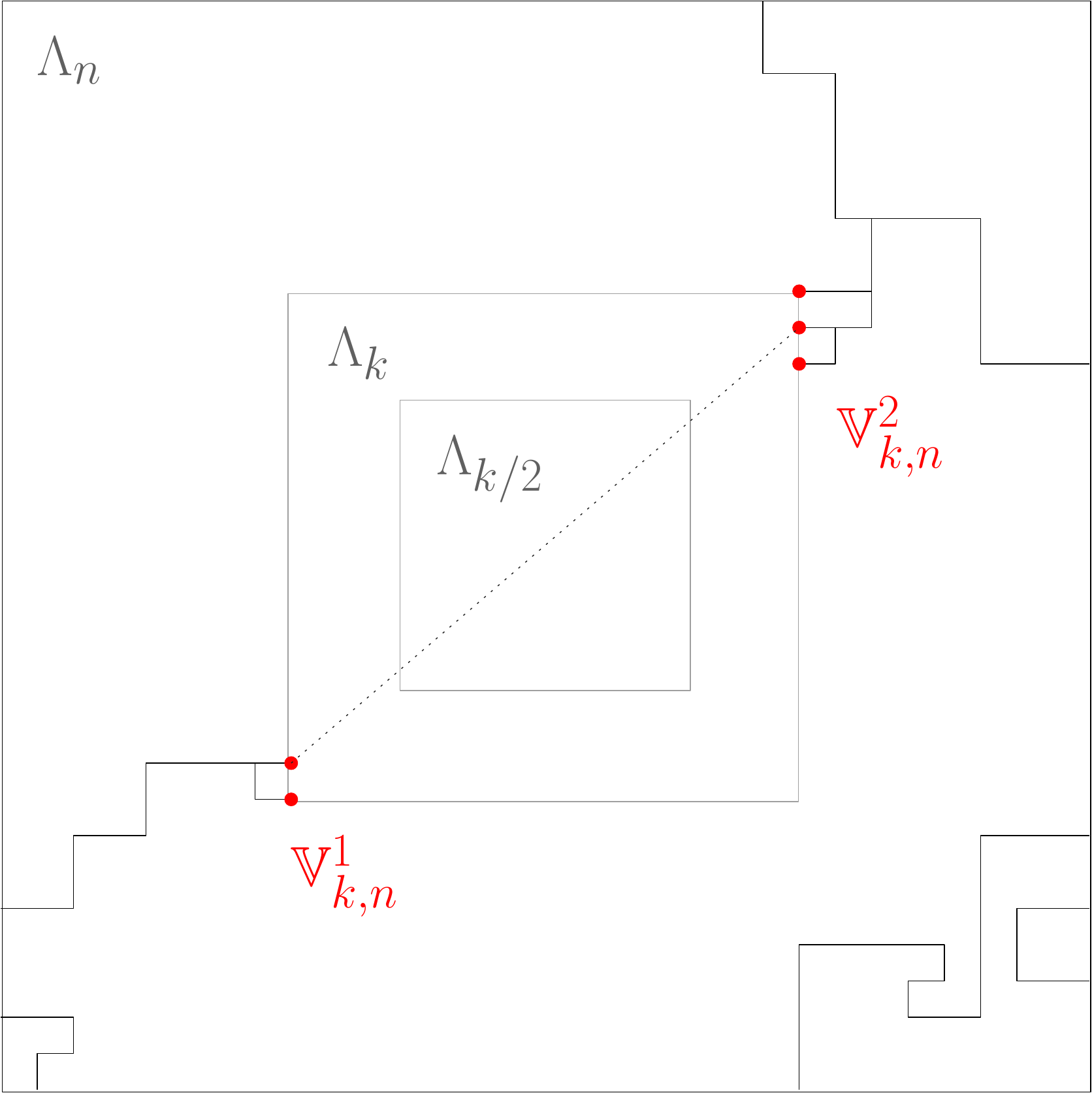}  \quad
\includegraphics[width=0.30\textwidth]{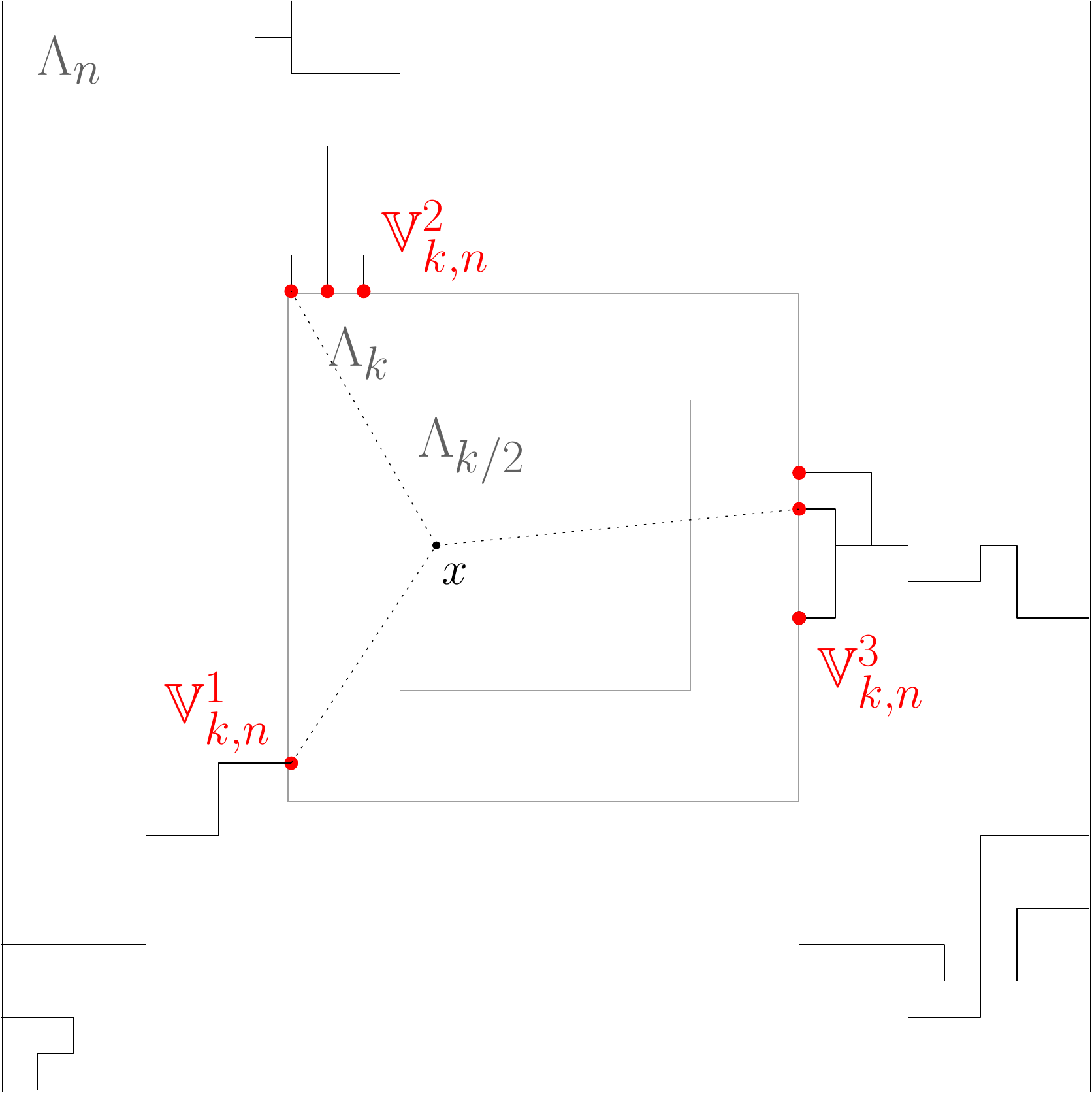}
\end{center}
\caption{Description of the events $E^\ell_{\nu,k}$, $\ell=1,2,3$ from left to
right. The set $\mathbb G_{k,n}$, partitioned into $\bbV^\ell_{k,n}$,
$\ell=1,2,3$, is indicated in red.}
\label{fig:3situations}
\end{figure}
\medbreak
We are now in a position to state the main proposition.
\begin{proposition}\label{tripod}
For any $\nu >0$, there exist $\delta =\delta(\nu,M)>0$ and $\kappa=\kappa
(\nu,M) > 0$ such that
\begin{equation}\label{tripod-flower-weird}
\FKff\Bigl(\ \bigcup_{k\ge\delta n} \big(E_{\nu ,k}^1\cup E_{\nu,k}^2\cup
E_{\nu,k}^3\big)\cap\big\{ |\bbG_{k,n}|\le M\big\}\ \Bigm|\
\Omega_{\underline{\bbG}}\ \Bigr) \ge 1- {\rm e}^{-\kappa n},
\end{equation}
uniformly in $(\calD,\ubbG)$ with $|\bbG|\le M$.
\end{proposition}
The proof of Proposition~\ref{tripod} comprises two steps: First, we show that
the implied geometric structure is characteristic of deterministic objects
called \emph{Steiner forests}. Then, we show that, with high $\FKff(\,\cdot\,|\,
\Omega_{\underline{\bbG}})$-probability, the cluster $\sfC_\bbG$ sits in the
vicinity of one such forest.

\subsection{Steiner forests}
\label{sub:Steiner}
{Note that 
for every $m$ the
set $\calK_m$ of all compact subsets of $\Lambda_m$ is a Polish space with
respect to the $\dd_\tau$-distance. } 

We now recall the concept of Steiner forest. Consider $E\subseteq
\partial\Lambda_m$ with $|E|\le M$. Let $\underline E=(E_1,\dots,E_i)$ be a
partition of $E$ and $\Omega_{\underline E}$ be the set of compact subsets of
$\bbR^2$ such that $E_j$ is included in one of their connected components for
every $j\in\{1,\dots,i\}$. For the trivial partition $\underline{E} = \lbr
E\rbr$, we shall write $\Omega_{E}$.
\bigbreak
For a compact $\calS \subset \bbR^2$, let $\tau (\calS )$ be the
(one-dimensional) Hausdorff measure of $\calS$ in the $\tau$-norm. Explicitly,
\begin{equation}
\label{eq:Functional-tau}
\tau (\calS)
=
\lim_{\varepsilon\to 0}\,\inf\Bigl\{ \sum {\rm diam}_\tau (A_i )\,:\, \calS
\subseteq\cup A_i,\, {\rm diam}_\tau (A_i )\leq \varepsilon\Bigr\},
\end{equation}
where ${\rm diam}_\tau(A)=\sup\{ \tau(x-y) : x,y\in A\}$. Define the set of
\emph{Steiner forests} by
\[
\Omega^{\rm min}_{\underline E}
=
\bigl\{ \calF\in\Omega_{\underline E} \,:\, \tau (\calF )
=
\min_{\calS\in\Omega_{\underline E}} \tau (\calS )\bigr\} .
\]
We set
\[
\tau_{\underline E}
=
\min_{\calS\in\Omega_{\underline E}} \tau (\calS)
=
\tau(\calF) ,
\]
for any Steiner forest $\calF\in \Omega^{\rm min}_{\underline E}$.

In the sequel we shall work only with Steiner forests 
$\calF\in \Omega^{\rm min}_{\underline E}$, when $\underline{E}$ is a 
partition of a set $E\subset \partial\Lambda_m$ of cardinality 
$\abs{E}\leq M$. Let $\Omega^{\rm min}_{M, m}$ be  the collection of all such 
forests. 

\begin{figure}
\begin{center}
\includegraphics[width=0.30\textwidth]{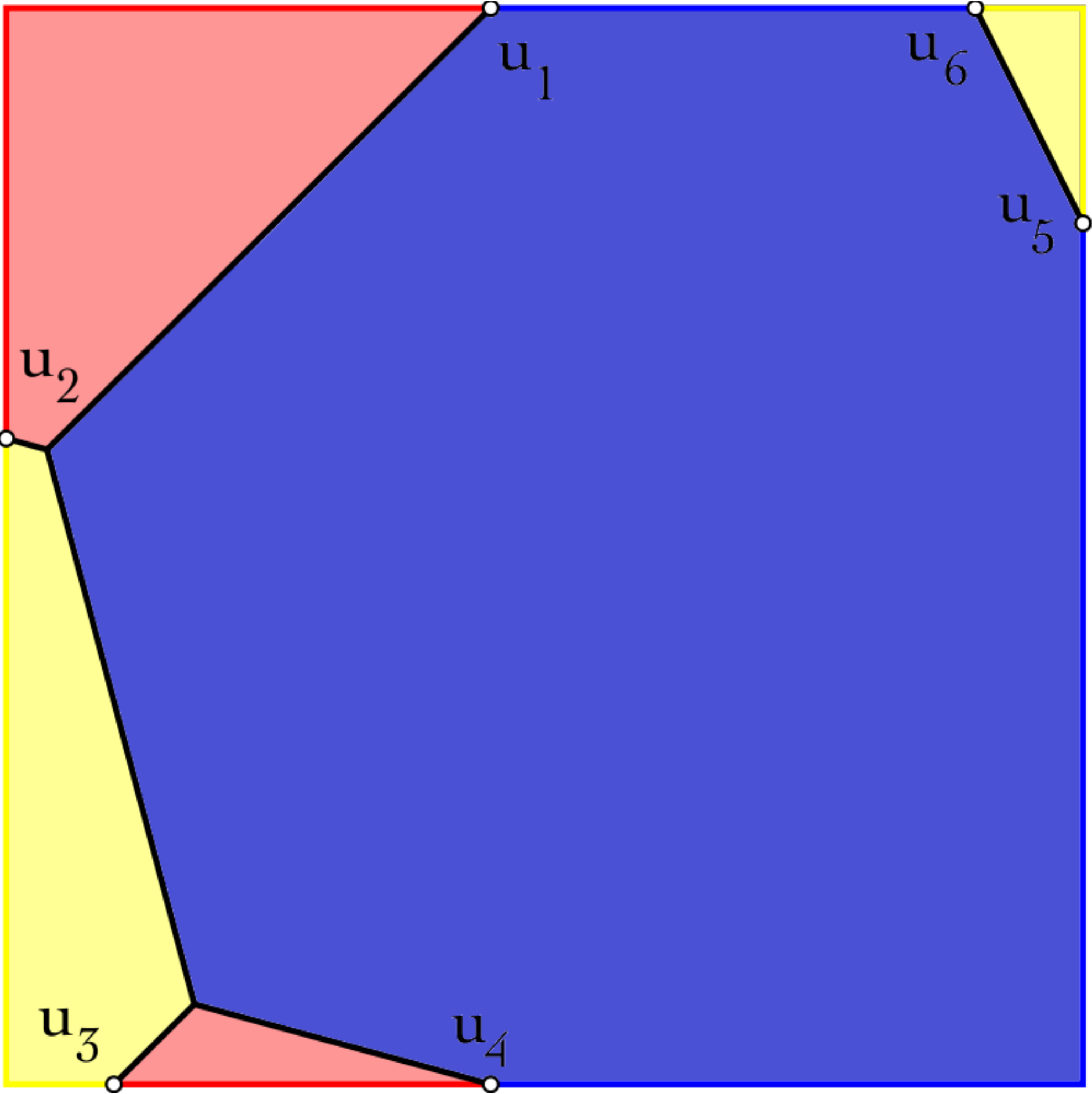}
\end{center}
\caption{A non trivial Steiner forest
with a partition $\underline{E} = (E_1, E_2)$ with $E_1 = \lbr u_1, 
\dots, u_4\rbr$ and $E_2 = \lbr u_5, u_6\rbr$.}
\label{fig:steiner_forest}
\end{figure}
\noindent
\begin{proposition}
\label{prop:Steiner}
Fix $M>0$.  The following properties hold uniformly in $m$, in finite subsets
$E\subseteq \partial \Lambda_m$ with $|E|\le M$ and in partitions $\underline E$
of $E$:
\begin{itemize}
\item[{\rm P1.}] \textbf{(Number of Steiner forests {and compactness of
$\Omega^{\rm min}_{M, m}$)}}
There exists $k=k(M)<\infty$ such that $|\Omega^{\rm min}_{\underline E} |\leq k$.
{ The set $\Omega^{\rm min}_{M, m}$ is a compact 
subset of $\lb \calK_m, \dd_\tau \rb$. }
\item[{\rm P2.}] \textbf{(Structure of Steiner forests)}
The sets $\calF\in\Omega^{\rm min}_{\underline E}$ are forests (that is
collections of disjoint trees). Each inner node (that is not belonging to $E$)
of such $\calF$ has degree $3$. Furthermore, there exists an $\eta>0$ such that
the angle between two edges incident to an inner node of $\calF$ is always
larger than $\tfrac\pi2 + \eta$.
\item[{\rm P3.}] \textbf{(Well separateness of trees)}
There exists $\delta_1=\delta_1(M) > 0$ such that  any $\calF\in\Omega^{\rm
min}_{\underline E}$ satisfies:
\begin{itemize}
\item[{\rm (a)}] for any Steiner tree $\calT\in\calF$, two different nodes  of
$\calT$ in $\Lambda_{m/2}$ are at $\dd_\tau$-distance at least $\delta_1 m$ of
each other;
\item[{\rm (b)}] if $\calT_1$ and $\calT_2$ are two disjoint trees of $\calF$,
then $\dd_\tau\lb \calT_1\cap \Lambda_{m/2} ,  \calT_2\cap \Lambda_{m/2}\rb 
\geq \delta_1 m$ .
\end{itemize}
\item[{\rm P4.}] \textbf{(Stability)}
For any $\delta_2 > 0$, there exists $\kappa_2=\kappa_2(\delta_2,M) > 0$ such
that,
for any $\abs{E}\leq M$, any partition $\underline{E}$ of $E$ and 
any $\calS\in \Omega_{\underline E}$,
\begin{equation}\label{eq:tau-distance}
\tau (\calS)
\leq
\tau_{\underline E} + \kappa_2 m\quad\text{implies}\quad
\min_{\calF\in\Omega^{\rm min}_{M, m}}{\rm d}_\tau (\calS, \calF)< \delta_2 m .
\end{equation}
\end{itemize}
\end{proposition}

\begin{proof}
We shall be rather sketchy since the arguments are presumably well understood.
We shall consider the case $m=1$ (the general case follows by homogeneity).  
\smallskip

Let us start with P4.
The functional $\tau$ in~\eqref{eq:Functional-tau} is lower semi-continuous on
$(\calK_1,\dd_\tau)$ and has compact level sets (meaning sets of the form
$\{\calS:\tau(\calS)\le R\}$). 
See, for instance, \cite[Proposition~3.1]{Co05}, where these facts are
explained for the inverse correlation length of sub-critical Bernoulli bond 
percolation. 

Assume that P4 is wrong. Then there exists $\delta  >0$ and  two sequences;  
$\underline{E_k}$ and $\calS_k\in \Omega_{\underline{E_k}}$, such that 
\[
\tau\lb \calS_j\rb < \tau_{\underline{E_k}}+\frac{1}{k}\quad \text{but}\quad 
\min_{\calF\in\Omega^{\rm min}_{M, m}}{\rm d}_\tau (\calS_j,\calF) > \delta .
\]
Since $\abs{E_k}\leq M$, the sequence $\tau_{\underline{E_k}}$ is bounded. 
Hence $\lbr \calS_j\rbr$ is precompact. Possibly passing to subsequence we may assume 
that $\underline{E_k}$ converges to $\underline{E}$ (points might 
collapse, but this is irrelevant since this preserves $\abs{E}\leq M$), 
and that $\calS_k$ converges to 
$\calS\in \Omega_{\underline{E}}$. Both convergence are, of course, in the 
sense of Hausdorff distance. By minimality it is evident that 
$\tau_{\underline E} = \lim\tau_{\underline{E_k}}$. 
By lower-semicontinuity $\tau (\calS )\leq \liminf \tau (\calS_k )$. 
Which means that $\calS\in \Omega_{\underline{E}}^{\rm min}$.
A contradiction.

A proof of {the first assertion} of  P1 can be found in~\cite[Theorem 1]{Coc67}.
Compactness of $\Omega_{M,1}^{\rm min}$ follows from 
compactness of level sets of $\tau$ and 
the fact that if 
$\calF_k\in \Omega_{\underline{E_k}}^{\rm min}$ converges to 
$\calF\in\Omega_{\underline{E}}$, then, as was already 
mentioned above, $\tau_{\underline E} = \lim\tau_{\underline{E_k}}$, and hence, 
by the lower-semicontinuity of $\tau$, 
$\calF\in \Omega_{\underline{E}}^{\rm min}$.

A proof of P2 can be found in~\cite{AlfCon91}.

\medbreak

Let us turn to the proof of P3.
For trivial partitions, Steiner forests are trees. Now, assume that there exists
a sequence of Steiner trees $\calT_k\in\Omega^{\rm min}_{{ E}_k}$ such that
$\calT_k$ contains at least two inner nodes in $\Lambda_{1/2}$ at distance less
or equal $\frac1k$. There is no loss of generality to assume that the sequence
$\calT_k$ converges to some $\calT^*$ . As it follows from P4,
$\calT^*\in\Omega^{\rm min}_{{E}^*}$, where ${E}^*$ is the corresponding limit
of ${E}_k$. Obviously, $|E^*|$ is still less or equal to $M$, since boundary
points can only collapse under the limiting procedure.

The total number of nodes of each of $\calT_k$ is uniformly bounded above. Hence
by our assumption we can choose  a number $\ell\geq 2$, a point
$x\in\Lambda_{1/2}$, a radius  $\varepsilon>0$ and a sequence $\nu (k )\to 0$,
so that \\
(a) each of $\calT_k$ contains $\ell$ nodes in $\Lambda_{\nu(k)}
(x)=x+\Lambda_{\nu(k)}$; \\
(b) none of $\calT_k$ contains nodes in the annulus  $A_{\nu (k),
\varepsilon}(x)$.
\\
Then the restriction of $\calT_k$ to $\Lambda_{\varepsilon} (x)$ is a Steiner
tree, whereas the cardinality of the intersection ${|\partial
\Lambda_{\varepsilon} (x)\cap \calT_k | = \ell+2}$. By the minimality of
$\calT_k$ the points of $\partial \Lambda_{\varepsilon} (x)\cap \calT_k $ are
uniformly separated. Consequently, $|\partial \Lambda_{\varepsilon} (x)\cap
\calT^* | = \ell+2 >3$. We infer that the {degree} of $x$ in  the Steiner tree
$\calT^*$ is $\ell +2 >3$, which is impossible by P2. This proves P3(a).

Consider now two disjoint Steiner trees $\calT_1\in \Omega^{\rm min}_{{ E}_1}$
and $\calT_2 \in \Omega^{\rm min}_{{ E}_2}$, such that the forest $\lbr\calT_1
,\calT_2\rbr$ belongs to $\Omega^{\rm min}_{\lbr { E}_1 , E_2 \rbr}$. By the
strict convexity of $\tau$, the trees are confined to their convex envelopes:
$\calT_i\in {\rm co}\lb E_i \rb$ for $i=1,2$. Thus if both trees are disjoint
and intersect $\Lambda_{1/2}$, it follows that ${\rm co}\lb E_1\rb \cap {\rm
co}\lb E_2\rb = \emptyset $. Consequently, there exist $u_1, v_1\in E_1$ and
$u_2, v_2\in E_2$, such that $\calT_1$ lies below the interval $[u_1 ,v_1]$  and
$\calT_2$ lies above the interval $[u_2 ,v_2]$ (notions of above and below are
with respect to the directions of normals). We are now facing two cases:
\begin{itemize}
\item $\calT_1$ or $\calT_2$ has an inner node in $\Lambda_{2/3}$. By P2, inner
nodes are of degree three and angles between edges incident to inner nodes are
at most $\pi -2\eta$. This pushes inner nodes of $\calT_i$ away from $[u_i ,v_i
]$ uniformly in $\calT_1$ and $\calT_2$. In such a case, P3 is satisfied.
\item  Both  $\calT_1$ and $\calT_2$ do not contain nodes in $\Lambda_{2/3}$,
but each contains an edge which crosses $\Lambda_{1/2}$. Having such edges close
to each other (and hence running essentially in parallel across $\Lambda_{1/2}$)
is easily seen to be incompatible with the minimality of $\calF$.
\end{itemize}
This achieves the proof of P3(b).
\end{proof}

\subsection{Forest skeleton  of the cluster $\sfC_\bbG$}

Let $\ubbG$ be a partition of $\bbG$.  We now aim to show that, under $\FKff
(\,\cdot\,|\, \Omega_{\ubbG} )$, the cluster $\sfC_\bbG$ stays typically close
to one of the Steiner forests from {$\Omega_{M, m}^{\rm min}$}. In order to do
that, we introduce the notion of forest skeleton of the cluster. This notion is
a  modification of the coarse-graining procedure developed in Section~2.2
of~\cite{CamIofVel08}.
\medbreak
Let $\U$ be the unit ball in $\tau$-norm. Fix a large number $c >0$  and
consider $K$ such that $c\log K<K$. For any $y\in\Z^2$, set
\[
\B_K (y)\, = \, \lb y+ K\cdot\U\rb\cap\Z^2  \quad\text{and}\quad
\BB_{K } (y)  \, =\,   \B_{K +c\log K}  (y).
\]
If $x\in A\subset \Z^2$ and $y\in A\cup\partial^{{{\rm ext}}} A$, we shall
use
$\{x\slra{A}y\}$ to denote the event that $x$ and $y$ are connected by an open
path from $x$ to $y$ whose vertices belong to $A$, with the possible exception
of the terminal point $y$ itself.
\medbreak
Let us construct the forest skeleton $\calF_{K}$ of the cluster $\sfC_\bbG$ (see
Figure~\ref{fig:coarse_graining}). Here and below, vertices in $\Z^2$ are
ordered using the lexicographical ordering. In the following construction, we
will often refer to the \emph{minimal} vertex having some property.
\bigbreak
\noindent
\step{1}. Set $r=1, i=1$. Set $x_0^1=u_{i_1}$ be the minimal vertex of $\bbG$.
Set $V = \lbr x_0^{1}\rbr$ and $\calC = \BB_K(x_0^1)$. Go to~\step{2}.
\bigbreak
\noindent
\step{2}. If there exists $x\in V$ and $u\in\bbG\setminus V$ such that
$u\in\BB_{2K} (x)$, then choose $u^*\in\bbG\setminus V$ to be the minimal such
vertex. Set $x^r_i = u^* ,{A_i^r = \B_K (x_i^r )}$  
and go to \step{3}. Otherwise, go to~\step{4}.
\bigbreak
\noindent
\step{3}. Update $V\rightarrow V\cup \{x_i^r\}$,  $\calC \to \calC \cup
\BB_K(x_i^r)$ and  $i \to i+1$. Go to~\step{2}.
\bigbreak
\noindent
\step{4}.
If there is at least one vertex $y\in\partial^{{\rm ext}} \calC$ such that
\[
y \slra{\sfC_\bbG\setminus\calC} \partial^{{\rm ext}} \B_K(y)
\setminus\calC ,
\]
then choose  $y^*$ to be the minimal such vertex, set $x_i^r =y^*, 
A_i^r = \B_K (x_i^r )\setminus \calC$,  and go to~\step{3}. Otherwise, go
to~\step{5}.
\bigbreak
\noindent
\step{5}. If $\bbG \subset V$, then terminate the construction. Otherwise,
choose $u^*$ to be the minimal vertex of $\bbG\setminus V$. Update $r\to r+1$
and set $x_0^r = u^*$. Update $V\to V\cup\lbr x_0^r\rbr$ and $i=1$. Go
to~\step{2}.
\bigbreak

\begin{definition} 
The above procedure produces $r$ disjoint sets of vertices $V^1 = \lbr x_0^1,
x_1^1, \dots\rbr$, $V^2 = \lbr x_0^2, x_1^2, \dots\rbr$, $\dots$, $V^r = \lbr
x_0^r, x_1^r, \dots\rbr$. The vertices $x_i^j$ constructed on~\step{4} 
are equipped with sets $A_i^j$, $j=1\ldots r$. Exit paths through 
such $A_i^j$-s contribute multiplicative factors ${\rm e}^{-K}$ each. 
Sets  $A_i^j$ for vertices $x_i^j$ constructed on \step{2} play no role and are
introduced for notational convenience only (see~\eqref{ar} below). By
construction, there are at most  $M$ such vertices. 

The edges within each group $\ell = 1, \dots , r$ are constructed as follows:
$x_i^\ell$ is connected to the vertex of
\[
\bigl\{ x_j^\ell \,:\,   j<i \text{ and } x_i\in \BB_{2K}(x_j)\bigr\}
\]
which has smallest index $j$. 

This produces a graph which is a set of $r$ trees
$\calT_{K}^1,\ldots,\calT_K^r$.
The union of the trees is called the \emph{forest skeleton} $\calF_K=\cup_\ell
\calT_K^\ell$.
\end{definition}
Note that we consider these graphs as compact subsets of $\bbR^2$.
An example of forest squeleton is drawn on Figure~\ref{fig:coarse_graining}.
The following result follows trivially from the construction of the forest
skeleton.
\begin{proposition}\label{prop:basic skeleton}
Let $\calF_K$ be the forest skeleton of $\sfC_\bbG$, then
\begin{enumerate}
\item $\bbG$ is included in the vertices of $\calF_K$.
\item Two vertices $u,v\in\bbG$ which were connected in $\sfC_\bbG$ are also
connected in $\calF_K$.
\item $\sfC_{\bbG}\subseteq \cup_{\ell,i} \BB_{2K} (x_i^\ell )$.
\end{enumerate}
\end{proposition}

\begin{figure}
\begin{center}
\includegraphics[width=0.65\textwidth]{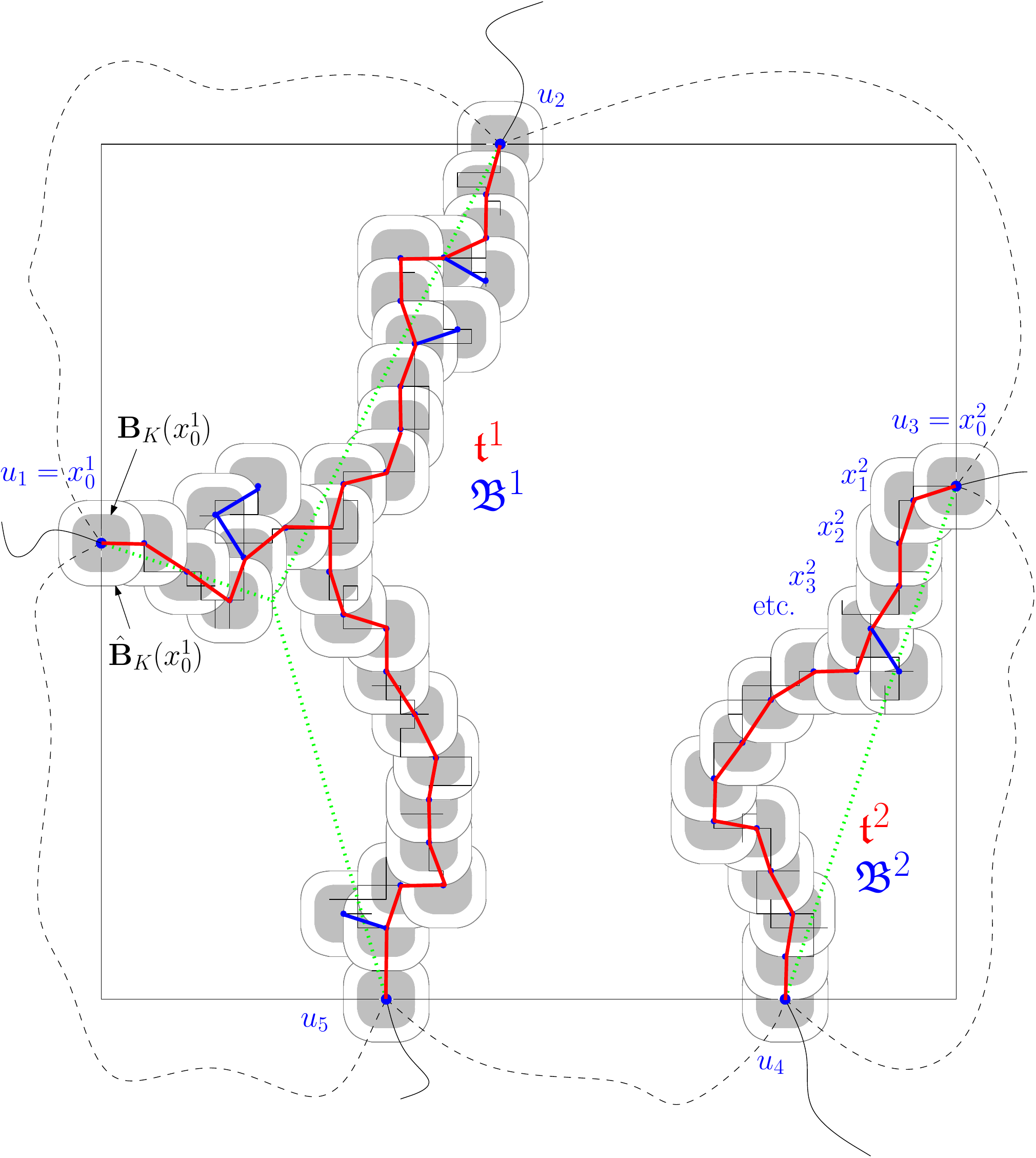}
\end{center}
\caption{{Construction of the forest skeleton
$\calF_{K}=\{\calT_K^1,\calT_K^2\}$ of the cluster $\sfC_\bbG$ (in black),
consisting of the trees $\calT_K^i=\{\frt^i,\frB^i\}$, $i=1,2.$ The Steiner
forest corresponding to the partition $\underline{\bbG}=(\{u_1,u_2,u_5\},\{u_3,u_4\})$ is
drawn in dashed green.}}
\label{fig:coarse_graining}
\end{figure}

\subsection{Distance between $\sfC_\bbG$ and Steiner forests}

\begin{proposition}\label{upper-away}
For every $\delta_3>0$, there exists $\kappa_3=\kappa_3(M)>0$ such that for $n$
large enough,
\[
\mu_{\calD}^\sff \Bigl( \min_{\calF\in\Omega^{\rm min}_{M,m }}{\rm d}_\tau
\bigl(\sfC_\bbG, \calF \bigr) > \delta_3 n
\Bigm| \Omega_{{\ubbG}} \Bigr)
\leq
{\rm e}^{- \kappa_3 n },
\]
uniformly in $(\calD,\ubbG)$ with $|\bbG|\le M$.
\end{proposition}
\begin{proof}
Let $\calF_K$ be the forest skeleton of $\sfC_\bbG$ at scale $K$ ($K$ will
be chosen later). By the third item of Proposition~\ref{prop:basic skeleton},
\[
\mathrm{d}_\tau\lb \sfC_\bbG , \calF_K\rb \leq 2K +c\log 2K.
\]
The proposition thus reduces to the following claim: for any $\delta_3>0$, there
exist $K=K(M)>0$ and $\kappa_3=\kappa_3(M)>0$ such that
\[
\mu_{\calD}^\sff \Bigl(  \min_{\calF\in\Omega^{\rm min}_{M,m }}
{\rm d}_\tau \bigl(\calF_{K}, \calF \bigr) > \delta_3 n
\Bigm| \Omega_{{\ubbG}} \Bigr)
\leq
{\rm e}^{- \kappa_3 n },
\]
uniformly in $(\calD,\ubbG)$ with $|\bbG|\le M$. We now prove this statement.
\bigbreak

Writing $E:=\{ \min_{\calF\in\Omega^{\rm min}_{M,m }} {\rm d}_\tau
\bigl(\calF_{K}, \calF \bigr) > \delta_3 n\}$, we have
\begin{align}
\mu_{\calD}^\sff (E\vert \Omega_{{\ubbG}} ) 
=
\frac{\mu_{\calD}^\sff (E \cap \Omega_{{\ubbG}} )}{\mu_{\calD}^\sff (
\Omega_{{\ubbG}} )}
\leq
\frac{\mu_{\calD}^\sff( \tau (\calF_{K }) \geq \tau_{{\bbG}} +\kappa_2 n)}
{\mu_{\calD}^\sff ( \Omega_{{\ubbG}} )}
\end{align}
where in the last inequality we used Property~P4 of
Proposition~\ref{prop:Steiner}, applied with $\delta_2 =\delta_3$. 

Let $\calF$ be a Steiner forest in $\Omega_{\ubbG}^{\rm min}$ and $\calF'$ be
the forest obtained by replacing each inner node of $\calF$ by the closest
vertex of $\Z^2$. Now, by the FKG inequality, we can lower bound the denominator
\begin{align}\label{lb-mg}
\mu_{\calD}^\sff (\Omega_{{\ubbG}} )
&\geq
\mu_{\calD}^\sff \Bigl(\bigcap_{\{x,y\}\in\calE(\calF')} \{x \leftrightarrow
y\}\Bigr)
\geq
\prod_{\{x,y\}\in\calE(\calF')}  \mu_{\calD}^\sff ( x \leftrightarrow y
)\nonumber\\
&\geq
\prod_{\{x,y\}\in\calE(\calF')}   {\rm e}^{-\tau(y-x)(1+o_{\abs{y-x}}(1))}
=
{\rm e}^{-\tau_{\bbG} (1+o_n(1))}.
\end{align}
where $\lim_{k\to\infty} \smo{k}=0$ by definition, and the product is taken
over the set $\calE(\calF')$ of all the inner edges of the approximate Steiner
forest $\calF'$.
\smallskip 

To obtain an upper bound on the numerator, we follow~\cite[Section
2]{CamIofVel08}. Let $\vert V(\calF_{{K}})\vert
=\sum_{\ell=1}^r \vert V^\ell\vert$ be
the total number of vertices of the forest skeleton $\calF$, then

\begin{align}\label{ar}
{{\rm e}^{-2KM}}\mu_{\calD}^\sff (\calF_{K}=\calF)
&\leq
 \mu_\calD^\sff\Bigl(\bigcap_{\ell=1}^r\bigcap_{i=0}^{\vert V^\ell\vert}
x_i^\ell
\stackrel{{A_i^\ell}}{\leftrightarrow}\partial^{{{\rm
ext}}}\B_K(x_i^\ell)\Bigr)
\leq
\prod_{\ell=1}^r\prod_{i=1}^{\vert V^\ell\vert} \mu_{\hat A_K^i}^\sfw\bigl(
x_i^\ell \stackrel{{A_i^\ell}}{\leftrightarrow}\partial^{{{\rm ext}}}
\B_K(x_i^\ell) \bigr)\nonumber
\\
&\leq
\bigl({\rm e}^{-K(1-\smo{K})}\bigr)^{\sum_{\ell=1}^r \vert V^\ell\vert}
=
{\rm e}^{-K\vert V(\calF)\vert (1-\smo{K})}\nonumber\\
&\leq
{\rm e}^{-\tau(\calF)(1-\smo{K}-\smo{n})},
\end{align}
where in the first inequality
{the term ${\rm e}^{-2MK}$ compensates (by the FKG inequality) the
inclusion
of events $x_i^\ell
\stackrel{{A_i^\ell}}{\leftrightarrow}\partial^{{{\rm
ext}}}\B_K(x_i^\ell)$ 
for points $x_i^\ell\in \bbG$, whereas in the second inequality}   
we expand the probability of the intersection as a
product of conditional expectations and then use the FKG inequality to compare
this conditional expectations with the probability with wired boundary
conditions, and in the second line we use that
$\mu_{\B_K(x)}^\sfw(x\leftrightarrow \partial^{{{\rm ext}}}
\B_K(x))={\rm e}^{-K(1-\smo{K})}$ (this
follows  from \cite[{Corollary 1.1}]{CamIofVel08}, which is now known
to be valid up to $p_c(q)$ thanks to Proposition~\ref{reference subcritical}).
If we now upper bound crudely the number of forest $K$-skeletons
rooted at $\bbG$ with $\tau(\calF)=T$ (and so with less than $C_1 T/K$ vertices)
by $(C_2 K)^{C_3 T/K}$, we get
\begin{align}
\nonumber
\mu_{\calD}^\sff( \tau (\calF_{K }) \geq \tau_{{\bbG}} +\kappa_2 n)
&=
\sum_{\calF : \tau(\calF)\geq \tau_\bbG+\kappa_2n} \mu_\calD^\sff(\calF_K=\calF)
=
\sum_{T\geq \tau_\bbG+\kappa_2n}\sum_{\substack{\calF \;:\\
\nonumber
\tau(\calF)=T}}\mu_\calD^\sff (\calF_K=\calF)\\
\nonumber
&=
\sum_{T\geq \tau_\bbG+\kappa_2n}{\rm e}^{(C_1 T/K)\log(C_2
K)-T(1-\smo{K}-\smo{n})}\\
\label{eq:tree-number}
&\leq
C_4 \cdot {\rm e}^{-(\tau_\bbG+\kappa_2 n)(1+\smo{K}+\smo{n})},
\end{align}
where we used~\eqref{ar} in the second line. The result follows by comparison
with~\eqref{lb-mg}:
\begin{align*}
\mu_\calD^\sff(E\vert\Omega_\bbG)
\leq
{\rm e}^{-(\tau_\bbG+\kappa_2 n)(1-\smo{K}-\smo{n})+\tau_\bbG(1+\smo{n})}
\leq
{\rm e}^{-n\kappa_3}.
\end{align*}
{Note that $\tau_\bbG o_n(1) = o(n)$ since $\tau_{\bbG}=O(n)$. The
latter} follows from the fact that $\tau_\bbG$ is bounded by the
$\tau$-length of the forest obtained by opening all the edges of $\partial E_m$
(recall that $\tau$ is an equivalent norm on $\R^2$).
\end{proof}

\begin{proof} [\textbf{Proof of Proposition~\ref{tripod}}]\ \\
Fix $\calD=\calD_{m,n}$ and $\bbG=\bbG_{m,n}$ with $m\ge \frac n3$ and
$|\bbG|\le M$. Let $\nu>0$. Fix an arbitrary $0<\delta\ll 1$ such that
$\Lambda_{250\delta n}\subset \delta_1 n\mathbf{U}_\tau$, where $\delta_1$ is
given by~P3. By definition of $\delta$, we know that for any forest $\calF \in
\Omega^{\rm min}_{{M}, m}$, $\calF\cap \Lambda_{250\delta n}$ is connected
and contains
at most one node. Therefore, we have three cases:
either $\calF\cap \Lambda_{2\delta n}=\emptyset$, or $\calF\cap
\Lambda_{2\delta n}\ne\emptyset$ but $\calF\cap \Lambda_{20\delta n}$ contains
only one edge, or $\calF\cap \Lambda_{20\delta n}$ contains more than one edge.
In the later case, the fact that edges incident to a node make an angle larger
or equal to $\frac{\pi}2+\eta$ implies that $\calF\cap \Lambda_{40\delta n}$
contains a node.
\medbreak

Also set $\delta_3<\min\{\nu,\delta \}$. Proposition~\ref{upper-away} implies
that
\begin{equation}\label{eq:Cluster-close1}
\min_{\calF\in\Omega^{\rm min}_{M,m }}\mathrm{d}_\tau \bigl( \sfC_\bbG,
\calF\bigr) \le \delta n,
\end{equation}
with probability larger than $1-{\rm e}^{-\kappa_3 n}$ for $n$ large enough. We
now assume that this inequality is indeed satisfied.
Since, by~P1 of Proposition~\ref{prop:Steiner} the set $\Omega^{\rm min}_{M,m }$
is compact, and since we are after an upper bound  which vanishes with $n$,
it will be enough to fix a Steiner forest $\calF \in \Omega^{\rm min}_{M,m }$
and to assume that 
\begin{equation}\label{eq:Cluster-close}
\mathrm{d}_\tau \bigl( \sfC_\bbG, \calF \bigr) \le \delta n,
\end{equation}
Let us treat the three previous cases separately.
\begin{itemize}
\item[C1.] $\calF\cap\Lambda_{2\delta n} = \emptyset$. In such case,
\eqref{eq:Cluster-close} shows that $\sfC_\bbG\cap\Lambda_{\delta n}=\emptyset$.
Thus, $E^1_{\nu,\delta n}$ holds true and $\bbG_{\delta n,n}=\emptyset$.
\item[C2.] $\calF\cap\Lambda_{20\delta n}=[u_1,u_2]$ with $u_1$ and $u_2$ on
$\partial\Lambda_{20\delta n}$ and  $[u_1 ,u_2]\cap \Lambda_{2\delta n}\neq
\emptyset$. In such case, \eqref{eq:Cluster-close} shows that $\sfC_\bbG$
intersects $\Lambda_{3\delta n}$ which in turns implies that $E^2_{\nu,k}$ holds
for every $k\in[6\delta n,18\delta n]$. Proposition~\ref{flowers-ndelta} implies
the existence of $k\in[6\delta n,18\delta n]$ with $|\bbG_{k,n}|\le M$
on an event of probability larger than $1-{\rm e}^{-18\delta n}$.
\item[C3.] There exists a node $x\in \Lambda_{40\delta n}$ and therefore
$\calF\cap\Lambda_{250\delta n}=[u_1,x]\cup[u_2,x]\cup[u_3,x]$ with $u_1$,
$u_2$, $u_3$ on $\partial\Lambda_{250\delta n}$ such that $\Angle{u_i - x}{u_j -
x} >\frac{\pi}{2} +\eta$ for every $i\neq j$. In such case,
\eqref{eq:Cluster-close} shows that $E^3_{\nu,k}$ holds for every $k\in[
82\delta n,246\delta n]$. Proposition~\ref{flowers-ndelta} implies the existence
of $k\in[82\delta n,246\delta n]$ with $|\bbG_{k,n}|\le M$
on an event of probability larger than $1-{\rm e}^{-246\delta n}$.
\end{itemize}
Altogether, we obtain the claim.
\end{proof}
For later use, let us introduce the following definition:
\begin{definition}
\label{rem:Tripod}
For $u_1, u_2, u_3$ in general position the function $\phi (y ) := \sum_{i=1}^3
\tau (u_i -y)$ is strictly convex and quadratic around its minimum point;
see~\cite[Lemma 3]{CamGia09}. Let $x$ be the unique minimizer of $\phi$.  In
this way the notation $\calT (u_1, u_2 , u_3 ; x)$ is reserved for the minimal
Steiner forest (in this case it is a tree) which contains $u_1, u_2, u_3$. It
might happen, of course, that $x$ coincides with one of the $u_i$-s.  When,
however, this is not the case, we shall refer to $\calT (u_1, u_2 , u_3 ; x)$ as
a \emph{Steiner tripod}.
\end{definition}

\section{Fluctuation theory and proof of Theorem~\ref{main RC}}
\label{sec:Fluctuations}
We are now in a position to prove Theorem~\ref{main RC}. Let $\nu>0$ small
enough to be fixed later. By~\eqref{eq:reduction2}
and~\eqref{tripod-flower-weird}, we can assume that there exist
$\delta=\delta(\nu)>0$ and $k\ge\delta n$ such that $|\bbG_{k,n}|\le M$ and
$E^\ell_{\nu,k}$ holds true for some $\ell\in\{1,2,3\}$. Let

\noindent
\[
\calR_n =
\max \left\{ k\geq\delta n : |\bbG_{k,n}|\le M\text{ and }E^1_{\nu,k}\cup
E^2_{\nu,k}\cup E^3_{\nu,k} \right\} \in\left[\delta n , n\right].
\]
Let $\calC$ be a possible realization of $\sfC_{k,n}$ and $\calD=\calD_{k,n}$ be
the corresponding flower domain.  We also set $\bbG=\bbG_{k,n}$. The restriction
of $\FK\lb\cdot~|~\calR_n =k ;\sfC_{k, n}=\calC\rb$ to $\calD$ is
$\FKflower$. Exactly as in Section~\ref{sec:redu},
\[
\cond\cap\{ \calR_n= k\} \cap \{ \sfC_{k,n} = \calC \}
=
\Omega_{\sigma,\calC}\times \{ \calR_n= k ; \sfC_{k,n}=\calC\},
\]
where $\Omega_{\sigma ,\calC}=\cup_{\underline{\bbG}\in
\calP'_\bbG}\Omega_{\underline{\bbG}}$ is defined as in Section~\ref{sec:redu}.
This reduction shows that it is sufficient to prove that
\[
\mu^\sff_\calD\bigl( \sfC_\bbG\cap\Lambda_{n^\varepsilon}\neq \emptyset \bigm|
\Omega_{\sigma,\calC}\bigr)
=
O(n^{\ep-1/2}),
\]
uniformly in the possible realizations of $\calD$, $\calC$ and $\bbG$.
\medbreak

From now on, we fix $k\ge \delta n$ such that $|\bbG_{k,n}|\le M$ and
$E^\ell_{\nu,k}$ holds true for some $\ell\in\{1,2,3\}$. We set
$\calD=\calD_{k,n}, {\calC = \calC_{k, n}}$ and $\bbG=\bbG_{k,n}$. 
\medbreak 
Since each set $\bbV^i$ is
already assumed to be connected outside of $\calD$ (since $E^1_{\nu,k}\cup
E^2_{\nu,k}\cup E^3_{\nu,k}$ occurs), partitions $\ubbG\in \calP'_{\bbG}$ can be
of four different types (recall that they are maximal in the sense defined in
the previous section): singletons only, singletons together with one pair of
elements in two different $\bbV^i$ (this cannot occur in $E^1_{\nu,k}$),
singletons together with one triplet of elements in three different $\bbV^i$
(this can occur only in $E^3_{\nu,k}$), singletons together with two pairs
$(u,v)$ and $(u',w)$, where $u$ and $u'$ belong to the same $\bbV^i$, and $v$
and $w$ belong to the other $\bbV^j$ (this can occur only in $E^3_{\nu,k}$). Let
$\calP^*_{\bbG}$ be the set of partitions in $\calP'_{\bbG}$ of one of the first
three  types.
If the configuration is in $\Omega_{\sigma,\calC}\setminus
\cup_{\ubbG\in \calP^*_{\bbG}} \Omega_{\ubbG}$, there are two different clusters
connecting two pairs of vertices $(u,v)$ and $(u',w)$ satisfying the conditions
described above. By choosing $\nu>0$ small enough, the assumption that
$E^3_{\nu,k}$ holds implies that $\tau(u-v)+\tau(u'-w)\ge
(1+\varepsilon)\tau_{\bbG}$ (where $\varepsilon=\varepsilon(\delta_3,\nu)>0$)
uniformly in the possible pairs $(u,v)$ and $(u',w)$. As in the proof of
Proposition~\ref{upper-away}, one obtains after a small computation that
\[
\mu^\sff_{\calD}\bigl( \Omega_{\sigma,\calC}\setminus \cup_{\ubbG\in
\calP^*_{\bbG}}\Omega_{\ubbG}\bigm| \Omega_{\sigma,\calC}\bigr)
=
O({\rm e}^{-ck}),
\]
for some constant $c>0$. Hence, a reduction in the spirit of
Proposition~\ref{prop:reduction2} shows that Theorem~\ref{main RC} would follow
from the bound
\begin{equation}
\label{eq:Three}
\mu^\sff_{\calD}\bigl( \sfC_{\bbG}\cap\Lambda_{n^\varepsilon}\neq \emptyset
\bigm| \Omega_{\ubbG}\bigr)
=
O(n^{\ep-1/2}),
\end{equation}
where the right-hand side is uniform in the possible realizations of $\calD$ and
in the $\ubbG\in \calP^*_{\bbG}$. We decompose the proof of~\eqref{eq:Three}
into three cases, depending on the type of $\ubbG$.

\paragraph{Scenario S1: No imposed crossing.}
This occurs in the following two cases (cf. Definition \eqref{E1E2E3}):
(i) $E^1_{\nu,k}$ occurs;
(ii) $E^2_{\nu,k}\cup E^3_{\nu,k}$ occurs and the partition $\ubbG$ is composed
of singletons only.
In this case, the measure
$\mu^\sff_{\calD_{k,n}}(\,\cdot\,|\Omega_{\ubbG})$ is unconditioned (i.e.
$\Omega_{\ubbG}=\Omega$). Proposition~\ref{reference subcritical} then implies
that $\mu^\sff_{\calD}\bigl( \sfC_{\bbG}\cap\Lambda_{n^\varepsilon}\neq
\emptyset \bigm| \Omega_{\ubbG}\bigr)$ decays exponentially with $n$.

\paragraph{Scenario S2: One imposed crossing.}
This occurs when $E^2_{\nu,k}\cup E^3_{\nu,k}$ occurs and $\ubbG$ is composed of
singletons together with a unique pair $(u,v)$, where $u\in\bbV^i, v\in\bbV^{j}$
with $i\ne j$. 
In other words, $\Omega_{\underline{\bbG}} =\lbr u\leftrightarrow v\rbr$. 
In this case, the cluster $\calC_{\bbG}\subseteq\calD$ may contain several
connected 
components, but, up to exponentially 
small (in $k$)  $\mu^\sff_\calD(\cdot |u\leftrightarrow v)$-conditional
probabilities, 
only one of them, namely the connected cluster $\sfC (u, v)$ of $\lbr u,v\rbr$
is  capable of reaching $\Lambda_{n^\varepsilon}$.
However, the law of the cluster connecting $u$ and $v$ converges to the law of
a Brownian bridge. In fact, one obtains the following stronger result:
\begin{equation}\label{eq:OZe}
\mu^\sff_\calD(x\in \sfC (u, v )|u\leftrightarrow v)
\le
\frac{C}{\sqrt{|u-v|}}\exp\Bigl(-\kappa
\frac{d_\tau(x,[u,v])^2}{|u-v|}\Bigr),
\end{equation}
where $\kappa$ and $C$ are constants depending on $p$ only, and $[u,v]$ denotes
the segment between $u$ and $v$. In the case of Ising interfaces, such bound was
obtained in~\cite[(3.31)]{GreIof05}. 
The proof relies on the positive curvature of the surface tension and on the 
effective random walk with exponentially decaying step distribution 
representation of the interface. 
The theory developed in \cite{CamIofVel08} enables a literal 
adaptation to the case of sub-critical FK-clusters, see Theorems~C and E  
and Subsections~4.4 and 4.5 in \cite{CamIofVel08}.
Consequently, 
\begin{equation}\label{eq:S2}
\FKff\bigl( \sfC_\bbG \cap \Lambda_{n^\varepsilon}\neq \emptyset \bigm|
\Omega_{\ubbG}\bigr) = O(n^{2\ep-1/2}).
\end{equation}

\paragraph{Scenario S3: One tripod.}
This can only happen when $E^3_{\nu,k}$ occurs and $\underline\bbG$ is composed
of singletons together with one triplet $(u_1,u_2,u_3)$ with
$u_1\in\bbV^1,u_2\in\bbV^2,u_3\in\bbV^3$.
Thus, in this case $\Omega_{\underline{\bbG}} = \lbr \sfC (u_1, u_2 , u_3 )\neq
\emptyset\rbr$, where $\sfC (u_1, u_2 , u_3 )$ is the joint connected cluster of
$\lbr u_1, u_2 , u_3\rbr$. Again, $\sfC = \sfC_{\bbG}\subseteq \calD$ may
contain several connected components, but, up to exponentially small (in $k$) 
$\mu^\sff_\calD(\cdot |\sfC (u_1 , u_2 , u_3 )\neq \emptyset)$-conditional
probabilities, only one of them, namely $\sfC (u_1 , u_2 , u_3 )$ itself, is 
capable of reaching $\Lambda_{n^\varepsilon}$.
By definition, there exists a unique 
$x{= x (u_1, u_2 , u_3 )} \in\Lambda_{k/2}$ (see
Definition~\ref{rem:Tripod}) such 
that $\calT_x = \lbr
u_1 ,u_2 ,u_3 ; x\rbr$ is a Steiner tripod. To lighten the notation, we set
\[
E(u_1,u_2,u_3,x)=\{u_1,u_2,u_3 \text{ are connected  and }
\mathrm{d}_\tau(\sfC_\bbG, \calT_x)\leq \nu k\}
\]
and redefine $\sfC = \sfC (u_1 , u_2 , u_3 ) $.
Thanks to Propositions~\ref{tripod} and~\ref{upper-away}, we now aim at proving
the bound 
\begin{equation}\label{eq:ThreeReduction}
\mu^\sff_{\calD}\bigl( {\mathsf C}\cap\Lambda_{n^{\ep}}
\neq\emptyset \bigm| E(u_1,u_2,u_3,x)\bigr)
=
O(n^{\ep-1/2}).
\end{equation}
This bound will imply Theorem~\ref{main RC}.

\smallbreak
Proving~\eqref{eq:ThreeReduction} is more complicated than
proving~\eqref{eq:S2}. Nevertheless, the idea remains the same: The tripod has
Gaussian fluctuations, therefore it intersects a small box with probability
going to 0. In the case of percolation, fluctuations of tripods on the level of
local limit results were studied in~\cite{CamGia09}. We are not after  a full
local limit picture here,  and merely explain how techniques
from~\cite{CamIofVel08} allow to derive~\eqref{eq:ThreeReduction}. Let us write
$\Lambda_r(x)$ for $x+\Lambda_r$.
\smallbreak
\begin{definition}\label{def:cones} (Cones $\calY_1,\calY_2,\calY_3$)\\
Since, {by Property P2 of the Steiner forests,} for every $i\ne j$,
\[
\Angle{u_i-x}{u_j-x} \ge \frac\pi2 + \eta,
\]
there exist disjoint cones $\calY_1, \calY_2$ and $\calY_3$ such that each
$\calY_i$ contains exactly one lattice direction in its interior (\emph{i.e.},
one of the four vectors $(1,0)$, $(0,1)$, $(-1,0)$ and $(0,-1)$, denoted by
$\sff_i$), and there exists $\ep_1>0$ such that $u_i\in {\rm int}\lb y+
\calY_i\rb$ for every $y\in \Lambda_{\ep_1 k }(x)$ and every $i\in\{1,2,3\}$,
and $u_i \in {\rm int}\lb u_j - \calY_i\rb$ for every $i\ne j$.
\end{definition}
\begin{definition}\label{def:S(t,y)} (Event $S(t,y)$)\\
Given $y\in \Lambda_{\ep_1 k} (x) $ and $t\in\bbN$, let  $\calS(t ,y)$ be the
event that the following three conditions occur:
\begin{itemize}
\item[{\rm R1.}]
$u_1$, $u_2$ and $u_3$ are pairwise disconnected in
$\sfC\setminus\Lambda_t(y)$,

\item[{\rm R2.}]
$\sfC$ intersects $\partial \Lambda_t (y )$ in exactly three
vertices.
\end{itemize}
For $i=1,2,3$, let $\sfC_i(t ,y)$ be the connected component of $\sfC\setminus
\Lambda_t (y )$ containing $u_i$, and $v_i(t ,y)=\sfC_i(t ,y)\cap \partial
\Lambda_t (y )$. Define $\sfC_0(t ,y)=\sfC\setminus\big(\sfC_1(t ,y)
\cup\sfC_2(t ,y) \cup\sfC_3(t ,y)\big)$. We will drop the reference to $t$ and
$y$ when no confusion is possible.

\begin{itemize}
\item[{\rm R3.}]
$\sfC_0$ is contained in $\bigcap_{i=1}^3\lb v_i - \calY_i\rb$ and 
$\sfC_i\subset \lb v_i +\calY_i\rb$ for $i\in\{1,2,3\}$.
\end{itemize}
\end{definition}

\begin{figure}
\begin{center}
\includegraphics[width=0.60\textwidth]{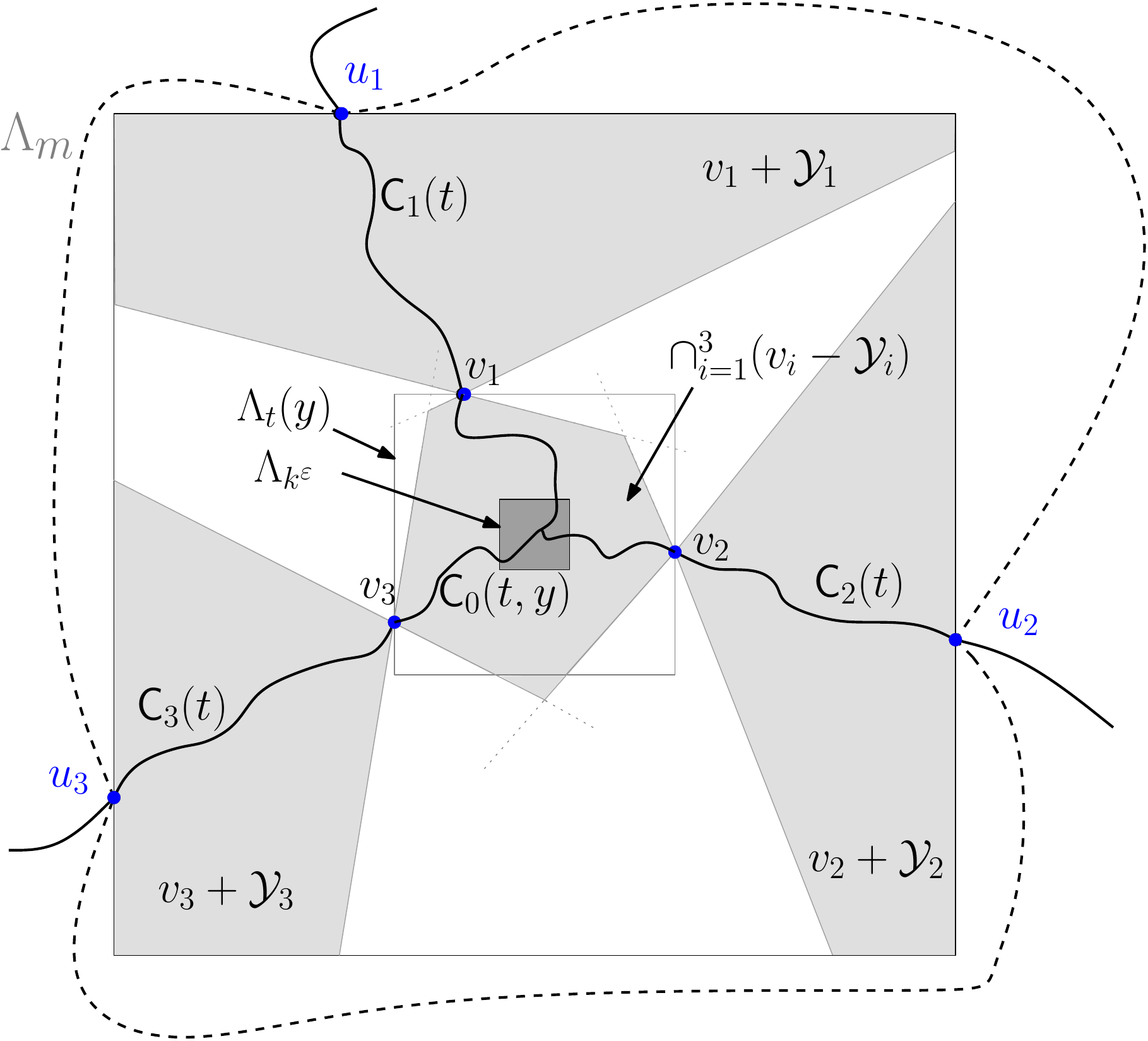}
\end{center}
\caption{Description of the event $S(t,y)$, namely the cones and the
decomposition of the cluster $\sfC$ into $\sfC_i(t)$ and $v_i(t)$, $i=1,2,3$.}
\label{fig:strongpivotal}
\end{figure}

\begin{lemma}\label{split-bound}
Fix $\ep_1 >0$ and let $\ep>0$ be sufficiently small.
There exists $C>0$ such that
\begin{equation}\label{eq:split-bound}
\mu_\calD^\sff\Bigl(\bigcup_{t\leq Ck^\ep} \bigcup_{ y\in \Lambda_{\ep_1 k}(x)}
\calS (t,y) \Bigm| E(u_1,u_2,u_3,x) \Bigr)
\ge
1-O({\rm e}^{-k^{\ep}}) .
\end{equation}
\end{lemma}
\begin{proof}[Proof of Lemma~\ref{split-bound}]
First of all, we notice that coarse-graining on the $k^\ep$-scale enables a
reduction to particularly simple geometric structures. Consider a forest
skeleton of the cluster $\sfC$ at scale $k^\ep$. Note that, conditionally on
$E(u_1,u_2,u_3,x)$, this forest is in fact a tree $\calT_{k^\ep}$. 

We define the
\emph{trunk} $\frt_\ep$ of $\calT_{k^\ep}$  as the
minimal subtree of $\calT_{k^\ep}$ which spans  $\lbr u_1, u_2, u_3\rbr$.

We define the \emph{branches} of $\calT_{k^\ep}$ as $\frB_\ep
=\calT_{k^\ep}\backslash \frt_\ep$. 
In this case, we obtain the following reduced geometry of typical
$\calT_{k^\ep}$, which holds uniformly in all situations in question, up to
probabilities which are exponentially small in $k^\ep$:
\begin{itemize}
\item[T1.]
$\calT_{k^\ep}$ does not have branches.
This means that the tree $\calT_{k^\ep}$ consists only of a trunk which
is a tripod, i.e. with one vertex of degree 3 and all other vertices of degree
at most 2. We will write $x_\ep$ for the only triple point of $\calT_{k^\ep}$,
and $\calT_{k^\ep}^i =\{ u_{i ,\ep}^{n_i}, \dots, u_{i ,\ep}^1 = x_\ep\}$,
$i=1,2,3$, for the three {legs}  of $\calT_{k^\ep}$. Note that $u_i\in
\hat\bfB_{2k^\ep}(u_{i,\ep}^{n_i})$.
\item[T2.]
Fix $\kappa > 0$ small. For every $\ep > 0$ and each $\epp\in (0,\ep/2)$, the
skeletons $\calT_{k^\epp}^i\setminus \Lambda_{k^\ep} (x_\epp )\subseteq x_\epp +
\calY_{i, 2\kappa} $ as soon as $k$ becomes sufficiently large, where cones
$\calY_{i, 2\kappa}$ are defined via
\begin{equation}
\label{eq:eta2}
\calY_{i, r} = \bigl\{ z : \Angle{z}{u_i-x_\epp} \leq r\bigr\}.
\end{equation}
That is, the vertices of each of the three branches of $\calT_{k^\epp}$ outside
the box $\Lambda_{k^\ep}$ are confined to the respective cones $x_\epp +
\calY_{i, 2\kappa}$.
\end{itemize}
Before proving Properties~T1 and~T2, let us describe how they can be used
to prove the lemma. First of all, note that, by Proposition~\ref{upper-away}, we
may assume that $|x_\epp - x|\leq \delta_1 k$ with $\delta_1 >0$ fixed as small
as we wish. In particular, we may assume that $u_i\in {\rm int} (x_\epp
+\calY_i)$ (see Definition~\ref{def:cones}) and, consequently, that $\calY_{i
,2\kappa}\subset \calY_i$.\\
By Proposition~\ref{prop:basic skeleton}, the connected cluster $\sfC$ is
included in $\calT_{k^\epp} + 2k^\epp {\bf U}_\tau$. Therefore, Properties~T1
and~T2 imply that $\sfC\setminus \Lambda_{k^\ep} (x_\epp )= \tilde\sfC_1
\cup\tilde \sfC_2 \cup \tilde\sfC_3$, where $\tilde \sfC_1$, $\tilde \sfC_2$ and
$\tilde \sfC_3$ are the clusters (in $\sfC\setminus \Lambda_{k^\ep} (x_\epp )$)
of $u_1$, $u_2$ and $u_3$ respectively. Note that, by~T2, clusters
$\sfC_i$ are confined to the sets (actually truncated cones) 
$( x_\epp + \calY_{i ,2\kappa} + 2k^\epp {\bf U}_\tau ) \setminus
\Lambda_{k^\ep} (x_\epp )$, which are  well separated on the  $k^\ep$-scale.
Consequently coarse-graining estimates
developed in~\cite[Section~2]{CamIofVel08} apply to each of $\sfC_i$ separately.
As a result, the claim of
Lemma~\ref{split-bound} follows by a straightforward adaptation of the 
mass-gap arguments of \cite[Section~2]{CamIofVel08}
applied separately to each of the three disjoint clusters $\tilde\sfC_1
,\tilde\sfC_2 $ and $\tilde\sfC_3$. For instance, one can show the following:
Fix $r$ large enough so that $\Lambda_{k^\epp} (x_\epp )\subset v- \calY_i$ for
any $v\in (x_\epp + \calY_{i ,2\kappa}) \cap (\Lambda_{2rk^\ep}
(x_\epp)\setminus \Lambda_{r k^\ep} (x_\epp))$ and $i=1,2,3$. Then, up to
probabilities which are exponentially small in $k^\ep$, there exists $t\in
[rk^\ep ,2rk^\ep]$ such that each of the clusters $\tilde\sfC_i$ contains a
$\calY_i$-break point on $\partial\Lambda_{t} (x_\epp)$. That is,
\begin{itemize}
\item for $i=1,2,3$, the intersection $v_i = \tilde\sfC_i\cap\partial \Lambda_t
(x_\ep )$ is a singleton;
\item for $i= 1,2,3$ the cluster $\tilde\sfC_i \subset (v_i +\calY_i) \cup (v_i
-\calY_i)$.
\end{itemize}
This ensures
$\calS (t ,y)$ for some $y\in \Lambda_{\ep_1 k}$ and~\eqref{eq:split-bound}
follows.
\end{proof}
\smallskip 

\noindent
For the proof of Property T1, we refer to~\cite[Lemma~2.1
and~2.2]{CamIofVel08}.
\begin{proof}[Proof of Property T2.]
Let us start with a lower bound on $\mu_\calD^\sff(E(u_1,u_2,u_3,x))$
which will be used later as a test threshold quantity for ruling out improbable 
events. 
Let $y$ be a lattice approximation of $x$. 
By the FKG inequality, 
\[
 \mu_\calD^\sff(E(u_1,u_2,u_3,x)) \geq 
 \mu_\calD^\sff \Bigl( \bigcap_{i=1}^{3} \bigl\{ y\stackrel{\calD}{\lra}
u_i\bigr\}\Bigr) 
 \geq \prod_{i=1}^3 
 \mu_\calD^\sff\bigl(  y\stackrel{\calD}{\lra} u_i\bigr) .
\] 
Theorem~A in \cite{CamIofVel08} gives sharp asymptotics of quantities $\mu^\sff
\lb 
y\lra v_i\rb$. These sharp asymptotics are built  upon an effective random walk 
representation of events $\lbr y\lra u \rbr$ as described in Subsection~4.1 of 
the the  paper. Steps of this random walk have effective drift from $u_i$
towards
$y$, and, since $\Lambda_k\subset \calD$, 
it is easy to adjust the arguments therein  in order  to show that 
\[
 \mu^\sff_\calD \bigl( y \stackrel{\calD}{\lra} u\bigr) \geq
\frac{C_0}{\sqrt{k}} {\rm e}^{-\tau (u -y )}, 
\]
uniformly in $y\in \Lambda_{\frac{k}{2}}$ and $v\in\partial\Lambda_k$, 
where $C_0$ (and, similarly, $C_1$, $C_2$, $\ldots$ below) is a universal
constant, in the sense that~\eqref{eq:lb-ep} applies uniformly in all the
situations in question as soon as $k$ is sufficiently large. Consequently, 
\begin{equation}\label{eq:lb-ep}
\mu_\calD^\sff(E(u_1,u_2,u_3,x))
\geq
\exp\Bigl( -\sum_{i=1}^3 \tau (u_i-x) - C_1 \log k \Bigr),  
\end{equation}
also uniformly in all the situations in question as soon as $k$ is sufficiently
large.  

Next, let us say that $\sfw\in\calT_{k^\epp}^i$ is a $2\kappa$-cone point of
$\calT_{k^\epp}^i$ if $\calT_{k^\epp}^i\subset \lb w-\calY_{i ,2\kappa}\rb \cup
\lb w+\calY_{i ,2\kappa}\rb$. 
In our notation, 
\[
 \tau (\calT_{k^\epp} ) = \sum_{i=1}^3 \tau (\calT_{k^\epp}^i )
\]
Since $\tau$ is a strictly convex norm (\cite[Subsection~1.3.2]{CamIofVel08}) ,
\begin{equation}
 \label{eq:taubound}
 \tau (\calT_{k^\epp}^i)  \geq \tau (u_i - x_\epp)\lb 1 +\delta (\kappa )\rb 
 \geq \tau (u_i - x_\epp)+ C_2  k, 
\end{equation}
whenever $\calT_{k^\epp}^i$ does not contain $2\kappa$-cone points at all.
This is essentially Lemma~2.4 of~\cite{CamIofVel08}. In view of~\eqref{ar}, and
in view of the lower bound~\eqref{eq:lb-ep}, 
we are entitled to ignore the situation when any of the $\calT_{k^\epp}^i$ does
not have $2\kappa$-cone points at all.

In the sequel, we use $\sfw^*_i$ to denote the first
$2\kappa$-cone point of $\calT_{k^\epp}^i$ (starting at $x_{\ep'}$) and $N_i$ to
denote its serial number; that is, $\sfw_i^* = u^{N_i}_{i,\epp}$.
Define $\calT_{k^\epp}^{i,*} 
= \{ u^1_{i, \epp}, \ldots , u^{N_i}_{i , \epp} = \sfw_i^*\}$ as the portion of
$\calT_{k^\epp}^i$ up to $\sfw_i^*$. 
Given $y$ and $\underline{\sfw} = \lb \sfw_1 , \sfw_2, \sfw_3\rb$,  define
the percolation event $E_\epp (y, \underline{\sfw} )\subset E(u_1,u_2,u_3,x)$
as 
\[
E_\epp(y, \underline{\sfw})= \bigl\{ x_\epp = y\, ;\, 
\sfw_i^* = \sfw_i\ \text{for $i=1,2,3$} \bigr\} .
\]
In view of~\eqref{eq:lb-ep}, Property~T2 will follow as soon as we shall have
checked that 
\begin{equation}
 \label{eq:T2-form}
  \mu^\sff_\calD \lb E_\epp(y, \underline{\sfw})\rb 
  \leq {\rm e}^{-\sum_i\tau (u_i -y) - C_3 k^\ep },
\end{equation}
uniformly in $k$,  tripods $\calT_x$, 
$y$ and $\underline{\sfw}\not\subset \Lambda_{k^\ep} (y )$.
For fixed realizations $\calT^{i, *}$ of $\calT^{i, *}_{k^\epp}$ we have 
\[
\mu^\sff_\calD \bigl( E_\epp(y, \underline{\sfw}); 
\calT^{i, *}_{k^\epp} = \calT^{i, *}\ \text{for $i=1,2,3$} \bigr)
\leq {\rm exp}\Bigl\{ -\sum_{i=1}^3\bigl\{ \tau (u_i - \sfw_i) +\tau
(\calT^{i,*}) (1-\smo{k^\epp})\bigr\} + C_4 k^{2\epp }\Bigr\} .
\]
This follows from~\eqref{ar} and from the finite energy property 
(applied for configurations on $\Lambda_{C_5 k^\epp} (\sfw_i )$) of the 
FK measures. Indeed, the finite energy property and the 
exponential ratio mixing property~\eqref{eq:ratio-mixing}
enable to decouple between the event $\bigcap_i \bigl\{ \calT^{i, *}_{k^\epp} =
\calT^{i, *}\bigr\}$ and the events $\bigcap_i \bigl\{ \sfw_i\stackrel{\sfw_i + 
\calY_{i ,2\kappa }}{\longleftrightarrow} u_i\bigr\}$.

Assume, for instance, that  $\sfw_1\not\in \Lambda_{k^\ep } (y)$.  
There are two cases to consider:

\case{1}: $\sfw_1 \in y + \calY_{1, \kappa}$. Then, exactly as
in~\eqref{eq:taubound}, $\tau (\calT^{1, *} )\geq \tau (\sfw_1-y ) +
C_6\abs{\sfw_1 - y}$. 
As in~\eqref{eq:tree-number} the entropic factor related to the number of
possible compatible realizations $\calT^{1 ,*}$ is suppressed,
and~\eqref{eq:T2-form} follows as soon as we choose $\ep>2\epp$. 
\smallskip 

\case{2}: $\sfw_1 \in ( y + \calY_{1, 2\kappa})\backslash ( y + \calY_{1,
\kappa})$.
By construction, $\tau (\calT^{1, *} ) \geq \tau (\sfw_1 - y )$. 
However, by the sharp triangle inequality~\eqref{eq:sharp-ti}, 
\[
\tau (\sfw_1 - y) +\tau (u_1 - \sfw_1 ) - \tau (u_1 -y ) \geq C_7\abs{\sfw_1
- y},
\]
uniformly in $\sfw_1$ under consideration. 
Again, since the entropic factor is suppressed, \eqref{eq:T2-form} follows.
\end{proof}

\begin{lemma}
\label{lem:stweights}
Let $S_\ep(y)=\bigcup_{t\leq Ck^\ep}S(t,y)$.
There exist two universal constants $\kappa>0$ and $C<\infty$ such that
\begin{equation}
\label{eq:stweights}
\mu_{\calD}^\sff \bigl( S_\ep(y) \bigm\vert E(u_1,u_2,u_3,x) \bigr)
=
O \Bigl( k^{12\ep-1} \exp\bigl(-\kappa
\frac{\abs{y -x}^2}{k}\bigr)\Bigr) ,
\end{equation}
uniformly in $y\in\Lambda_{\ep_1 k} (x)$.
\end{lemma}
\begin{proof}
Decompose
\begin{eqnarray*}\label{decS2}
\calS( t, y ) = \bigcup_{W}\calS^{W} (t,y)
\end{eqnarray*}
according to the triple $W = \{ v_1-y ,v_2-y ,v_3-y \} \subset
\partial\Lambda_t$ which shows up in the definition. From now on, we set
$w_1=v_1-y$, $w_2=v_2-y$ and $w_3=v_3-y$.

Since, under the event $S^W(t,y)$, we have that
$\sfC_0\subseteq\bigcap_{i=1}^3\lb v_i - \calY_i\rb$ and that the points $u_i$
lie deep in the interior of the corresponding cones $v_i+\calY_i$, with
$v_i\in\partial\Lambda_t(y)$ and $t\leq Ck^\ep$, the Ornstein-Zernike
asymptotics of~\cite[Theorem A]{CamIofVel08} imply that
\begin{equation}\label{eq:ac}
\mu_{\calD}^\sff \Bigl( \bigcap_{i=1}^3\{ \sfC_i\subset v_i +\calY_i \}  \Bigm|
\sfC_0(t,y)\Bigr)
=
\Theta\bigl(k^{-3/2}\, {\rm e}^{-\sum_{i=1}^3 \tau (u_i-v_i)}\bigr),
\end{equation}
uniformly in any possible realization $\sfC_0$  of
$\sfC_0(t,y)$ compatible with $S^W(t,y)$.  
Note that if $\sfC_0(t,y)$ is compatible with $S^W(t,y)$, then  shifts
$\sfC_0^{u}\df \sfC_0 +u$ are compatible with shifted events $S^W(t,y+u)$.

Recall Definition~\ref{rem:Tripod} of $\phi ( y)$. Given a triple 
$W = \lbr w_1 , w_2 , w_3\rbr \subset \Lambda_{Ck^\ep}$, let us define
\[
\phi_W (y) = \sum_{i=1}^3 \tau (u_i-w_i-y).
\]
Together with~\eqref{eq:ac}, we obtain
\begin{align}
\label{eq:Ratios}
\frac{ \mu_{\calD }^\sff\bigl(\calS^{W}(t,y) \bigr) } { \mu_{\calD
}^\sff \bigl( \calS^{W}(t,z) \bigr) } %\leq
&=
\frac{\sum_{\sfC_0}\mu_{\calD}^\sff
\Bigl( \bigcap_{i=1}^3\{ \sfC_i\subset y+w_i +\calY_i \}  \Bigm|
\sfC_0\Bigr)\mu_\calD^\sff(\sfC_0(t,y)=\sfC_0)}
{\sum_{\sfC_0}\mu_{\calD}^\sff \Bigl( \bigcap_{i=1}^3\{ \sfC_i\subset z+w_i
+\calY_i \}  \Bigm|
\sfC_0\Bigr)\mu_\calD^\sff(\sfC_0(t,z)=\sfC_0^{{z-y}})}\nonumber\\
&=
\Theta\bigl({\rm e}^{ \phi_W (y) - \phi_W (z) }\bigr) ,
\end{align}
uniformly in $t\leq Ck^\ep$, $W = \lbr w_1 ,w_2 ,w_3\rbr\subset
\partial\Lambda_t$ and $y,z\in \Lambda_{\ep_1 k}$, where the sum is
over $\sfC_0$ compatible with $S^W(t,y)$ and where in the last line we
used classical ratio-mixing properties of subcritical random-cluster
measures~\cite[Theorem 3.4]{Ale98} {and \eqref{eq:ratio-mixing}} 
to compare
$\mu_\calD^\sff(\sfC_0(t,y)=\sfC_0)$ and
$\mu_\calD^\sff(\sfC_0(t,z)=\sfC_0^{z-y})$.

The function $\phi_W$  has a non-degenerate quadratic minimum at
some $x_{\rm min}(W)$ (see~\cite[Lemma 3]{CamGia09}). 
In view of the homogeneity of $\tau$, a quadratic expansion around $x_{\rm min}$
yields
\begin{equation}\label{eq:quadratic0}
\phi_W (y) - \phi_W (x_{\rm min})
=
\Theta\Bigl( \frac{\abs{y - x_{\rm min}}^2}{k}\Bigr) ,
\end{equation}
uniformly in all situations in question. 
Since $|\phi(y)-\phi_W(y)|=O(k^\ep)$, its minimizers $x_{\rm min} (W)$ solve 
\[
  F(x, W ) \df \nabla_x\phi_W (x ) = 0.
\] 
Since ${\rm  Hess} (\phi )$ is non-degenerate at $x$, the implicit function 
theorem applies. As a result,
$|x_{\rm min} (W) -x|=O(\sum_i \abs{w_i }) = O (k^\ep )$ uniformly in all $W$ in 
question.
This fact, together with~\eqref{eq:quadratic0}, yields
\begin{equation}\label{eq:quadratic1}
\phi_W (y)-{\phi_W(x ) }
=
\Theta\Bigl( \frac{\abs{y - x}^2}{k}+k^{2\ep-1}\Bigr) .
\end{equation}
Since there are at most $O(k^{3\ep})$ possible choices for $W$ and $O(k^\ep)$
possible choices for $t$, we deduce from~\eqref{eq:Ratios}
and~\eqref{eq:quadratic1} that
\begin{equation}
\label{eqr}
\frac1{O(k^{4\ep})} \exp\bigl(-{C_1}\frac{|y-x|^2}{k}\bigr)
\leq
\frac{ \mu_{\calD }^\sff ( \calS_\ep(y) )} { \mu_{\calD }^\sff (
\calS_\ep({x} ) ) }
\leq
O(k^{4\ep}) \exp\bigl(-{C_2}\frac{|y-x|^2}{k}\bigr) .
\end{equation}
Above $\calS_\ep({x} )$ means in fact $\calS_\ep(\lfloor x\rfloor )$.
We can now compute
\begin{align}\label{ratio}
\mu_\calD^\sff(S_\ep(y) \,\vert\,
E(u_1,u_2,u_3,x))
&=
\frac{\mu_\calD^\sff(S_\ep(y))}{\mu_\calD^\sff(S_\ep(x 
))}\cdot\frac{\mu_\calD^\sff(S_\ep(x ))}{\mu_\calD^\sff(E(u_1,u_2,u_3,x))}
\nonumber\\
&\leq
O(k^{4\ep})\exp\Bigl(-C_2\frac{|y-x|^2}{k}\Bigr)\frac{\mu_\calD^\sff(S_\ep(x
))}{\mu_\calD^\sff(E(u_1,u_2,u_3,x))}
\end{align}
where we used the second inequality in~\eqref{eqr}. In order to see that the
rightmost term in~\eqref{ratio} is of the right order, observe that $|y-x|\le
k^{1/2-\ep}$ implies that ${\rm e}^{-C_2|y-x|^2/k}$
is of order 1, and therefore the ratio in~\eqref{eqr} is smaller than
$O(k^{4\ep})$. Therefore, by looking at the $k^{1-2\ep}$ sites which are at
distance at most $k^{1/2-\ep}$ from $x$, we deduce, using the first inequality
in~\eqref{eqr}, that
\[
\mu_\calD^\sff(S_\ep(x ))
\leq
O(k^{-1+6\ep}) \sum_{y\in\Lambda_{k^{1/2-\ep}}(x )}
\mu_\calD^\sff(S_\ep(y))
\leq
O(k^{-1+8\ep})\mu_\calD^\sff(E(u_1,u_2,u_3,x)),
\]
where in the second inequality, we used the fact that in a given configuration
there are at most $O(k^{2\ep})$ sites $y$ such that the corresponding events
$\calS_\ep(y)$ occur. This implies that
\[
\frac{\mu_{\calD}^\sff \bigl( \calS_\ep(x)
\bigr)}{\mu^\sff_\calD\bigl(E(u_1,u_2,u_3,x)\bigr)}
\leq
O(k^{8\ep-1}).
\]
Together with~\eqref{ratio}, we obtain~\eqref{eq:stweights}.
\end{proof}
Lemmas~\ref{split-bound} and~\ref{lem:stweights} imply that
\begin{multline}
\label{eq:ysum}
\mu_{\calD }^\sff \bigl( \sfC\cap \Lambda_{k^\ep}\neq \emptyset \bigm|
E(u_1,u_2,u_3,x) \bigr)
\leq \\
O(k^{12\ep-1}) \sum_{y\in\Lambda_{\ep_1 k} (x)} {\rm e}^{-\kappa \abs{y-x}^2/k}
\mu_{\calD}^\sff \bigl( \sfC\cap\Lambda_{k^\ep}\neq \emptyset \bigm|
\calS_\ep(y) \bigr) + O\bigl( {\rm e}^{-k^\ep} \bigr) . 
\end{multline}
It remains to provide an upper bound on $\mu_{\calD}^\sff\bigl(\sfC\cap
\Lambda_{k^\ep}\neq \emptyset \bigm| \calS_\ep(y)\bigr)$. There are two cases to
consider:
\medbreak
\noindent
\case{1}: $y\in\Lambda_{2Ck^\ep}$. In this case, we simply use
\begin{equation}\label{eq:case1}
\mu_{\calD}^\sff\bigl( \sfC\cap\Lambda_{k^\ep}\neq\emptyset \bigm| \calS_\ep (y)
\bigr) \leq 1.
\end{equation}
The total contribution to the right-hand side of~\eqref{eq:ysum} is then bounded
by $O(k^{14\ep-1})$, which is negligible with respect to our target
estimate~\eqref{eq:ThreeReduction}.
\medbreak
\noindent
\case{2}: $y\not \in \Lambda_{2Ck^\ep}$. In this case, $\Lambda_{k^\ep}$ can
intersect at most one of the $\Lambda_{Ck^\ep}(y) + \calY_i$, and therefore can
be hit by only one cluster $\sfC_i$.  {Without loss of generality, let us assume
that $\sfC_i=\sfC_1$. Conditioning on the smallest $t$ such that $\calS(t,y)$
occurs as well as on $\sfC_0$, $\sfC_2$ and $\sfC_3$, the cluster $\sfC_1$ obeys,
as was explained after \eqref{eq:OZe}, a diffusive scaling. In particular,
\[
\mu_{\calD }^\sff\bigl(z \in \sfC_1 \bigm| S(t,y),\sfC_0,\sfC_2,\sfC_3 \bigr)
=
O\Bigl(\frac{1}{\sqrt{|v_1-z|}} \, \exp\Bigl[ -\kappa' \frac{\dd_\tau
(z,[v_1,u_1])^2}{\abs{v_1-z}}\Bigr]\Bigr).
\]
In the previous inequality,  $v_1=v_1(t,y)$.
We find: 
\begin{align}
\mu_{\calD }^\sff \bigl( \sfC\cap\Lambda_{k^\ep}\neq \emptyset \bigm|
S(t,y),\sfC_0,\sfC_2,\sfC_3 \bigr)
&\leq
\sum_{z\in \partial \Lambda_{k^\ep}}\mu_{\calD }^\sff\bigl(z \in \sfC_1 \bigm|
S(t,y),\sfC_0,\sfC_2,\sfC_3 \bigr)\nonumber\\ 
&\leq
\sum_{z\in \partial \Lambda_{k^\ep}} O\Bigl(\frac{1}{\sqrt{|v_1-z|}} \,
\exp\Bigl[ -\kappa' \frac{\dd_\tau (z,[v_1,u_1])^2}{\abs{v_1-z}}\Bigr]
\Bigr)\nonumber\\
&=
O\Bigl(\frac{k^\ep}{\sqrt{\abs{y}}} \, \exp\bigl\{ -\kappa^{\prime\prime} \frac{\dd_\tau
(0,[y,u_1])^2}{\abs{y}}\bigr\} \Bigr)
\label{eq:Clusterhit}.
\end{align}
In the last line, we used the fact that $y, v_1\notin \Lambda_{Ck^\ep}$ and
$|v_1-y|\le Ck^\ep$. Let us substitute~\eqref{eq:Clusterhit} into the sum
on the right-hand side of~\eqref{eq:ysum} to obtain
\begin{multline}
\mu_{\calD }^\sff \bigl( \sfC\cap \Lambda_{k^\ep}\neq \emptyset \bigm|
E(u_1,u_2,u_3,x) \bigr) \leq O({\rm
e}^{-k^\varepsilon})+O(k^{14\varepsilon-1})\\
+ O(k^{13\ep-1}) \sum_{y\in\Lambda_{\ep_1 k} (x)\setminus \Lambda_{2C k^{\ep}}}
\frac{1}{\sqrt {|y|}}\exp\Big[-\kappa
\frac{\abs{y-x}^2}{k}-\kappa^{\prime\prime}\frac{\dd_\tau (0,[y,u_1])^2}{|y|}\Big].
\end{multline}
After a simple estimate, one sees easily that the sum on the right is 
bounded above as
\[
2\sum_{y\in\Lambda_{k^{1/2+\ep}}(x)\setminus \Lambda_{2C k^{\ep}}} \frac{1}{\sqrt
{|y|}}\exp\Big[-\kappa^{\prime\prime}\frac{\dd_\tau (0,[y,u_1])^2}{|y|}\Big]
\leq
\sum_{\ell=\max\{2Ck^\ep,|x|-k^{1/2+\ep}\}}^{|x|+k^{1/2+\ep}} \frac{O(\sqrt
\ell)}{\sqrt \ell}=O(k^{\frac12+\ep}),
\]
uniformly in $x$. In order to obtain the first inequality, we used the fact that
$\exp(-\kappa |y-x|^2/k)$ is very small for sites outside of
$\Lambda_{k^{1/2+\ep}}(x)$. For the second, observe that sites $y$ at distance
$\ell$ contributing substantially to this sum must satisfy the condition that
$0$ is at distance $O(\sqrt{\ell})$ of $[y,u_1]$. There are $O(\sqrt{\ell})$ of
them. This concludes the proof.
  
%%%%%%%%%%%%%
%%%%%%%%%%%%%
 
\paragraph{Acknowledgements.} H. D.-C. was supported by the EU Marie-Curie RTN
CODY, the ERC AG CONFRA, as well as by the Swiss National Science Foundation. The research of D.I. was
supported by the Israeli Science Foundation (grant No 817/09). L.C. and Y.V.
were partially supported by the Swiss National Science Foundation.

\begin{flushright}
\footnotesize\obeylines
  \textsc{D\'epartement de Math\'ematiques}
  \textsc{Universit\'e de Gen\`eve}
  \textsc{Gen\`eve, Switzerland}
  \textsc{E-mail:} \texttt{loren.coquille@unige.ch; hugo.duminil@unige.ch; yvan.velenik@unige.ch}

  \textsc{Faculty of IE\&M}
  \textsc{Technion}
  \textsc{Haifa, Israel}
  \textsc{E-mail:} \texttt{ieioffe@ie.technion.ac.il}
\end{flushright}
\end{document}